\documentclass[floatfix,pre,superscriptaddress,twocolumn, aps, 10pt]{revtex4-2}
\usepackage{hyperref}
\usepackage{graphicx}
\usepackage{amsmath}
\usepackage{amsfonts}
\usepackage{amssymb}
\usepackage{amsthm}
\usepackage{cancel}

\usepackage{algorithm}
\usepackage{algpseudocode}
\usepackage{esvect}
\usepackage{bm}
\usepackage{appendix}
\usepackage{soul}
\usepackage{mathtools, nccmath}
\usepackage{hyperref}
\usepackage{xcolor}
\usepackage{overpic}

\makeatletter
\renewcommand{\fps@algorithm}{htbp}
\renewcommand{\ftype@algorithm}{1}
\renewcommand{\ext@algorithm}{lof}
\makeatother

\DeclareMathOperator{\Tr}{Tr}
\DeclareMathOperator*{\argmin}{argmin}

\begin{document}

\title{Decomposition Pipeline for Large-Scale Portfolio Optimization with Applications to Near-Term Quantum Computing}

\author{Atithi Acharya}
\author{Romina Yalovetzky}
\author{Pierre Minssen}
\author{Shouvanik Chakrabarti}
\author{Ruslan Shaydulin}
\author{Rudy Raymond}
\author{Yue Sun}
\author{Dylan Herman}
\affiliation{Global Technology Applied Research, JPMorganChase, New York, NY 10017, USA}

\author{Ruben S.~Andrist}
\affiliation{Amazon Quantum Solutions Lab, Seattle, Washington 98170, USA}

\author{Grant Salton}
\affiliation{Amazon Quantum Solutions Lab, Seattle, Washington 98170, USA}
\affiliation{AWS Center for Quantum Computing, Pasadena, CA 91125, USA}
\affiliation{California Institute of Technology, Pasadena, CA, USA}

\author{Martin J.~A.~Schuetz}
\affiliation{Amazon Quantum Solutions Lab, Seattle, Washington 98170, USA}
\affiliation{AWS Center for Quantum Computing, Pasadena, CA 91125, USA}

\author{Helmut G.~Katzgraber}
\affiliation{Amazon Quantum Solutions Lab, Seattle, Washington 98170, USA}

\author{Marco Pistoia}
\affiliation{Global Technology Applied Research, JPMorganChase, New York, NY 10017, USA}

\date{\today}

\begin{abstract}
Industrially relevant constrained optimization problems, such as portfolio optimization and portfolio rebalancing, are often intractable or difficult to solve exactly. In this work, we propose and benchmark a decomposition pipeline targeting portfolio optimization and rebalancing problems with constraints. The pipeline decomposes the optimization problem into constrained subproblems, which are then solved separately and aggregated to give a final result. Our pipeline includes three main components: preprocessing of correlation matrices based on random matrix theory, modified spectral clustering based on Newman's algorithm, and risk rebalancing. Our empirical results show that our pipeline consistently decomposes real-world portfolio optimization problems into subproblems with a size reduction of approximately 80\%. Since subproblems are then solved independently, our pipeline drastically reduces the total computation time for state-of-the-art solvers. Moreover, by decomposing large problems into several smaller subproblems, the pipeline enables the use of near-term quantum devices as solvers, providing a path toward practical utility of quantum computers in portfolio optimization.

\end{abstract}

\maketitle

\section{Introduction}

Portfolio optimization (PO) plays a vital role in managing vast amounts of global financial assets, requiring rapid and robust algorithms to effectively make decisions about investing capital, while managing risks. Often, PO problems include integer decision variables, which arise when the underlying assets are traded in discrete quantities. Moreover, such optimization problems impose constraints on the decision variables, such as having a fixed minimum and maximum number of selected assets, or a maximum total risk exposure of the portfolio. 

Typically, PO problems can be formulated as mixed-integer programming (MIP) problems, for which a wide range of methods have been proposed. In particular, various local search methods have been proposed and successfully applied~\cite{vaessens1998local, PhysRevE.92.013303, Hauke_2020}.
However, these methods often have several drawbacks when applied to MIP problems, such as scalability, the risk of getting trapped in local optima, and difficulty handling complex constraints. On the other hand, exact methods, such as those based on the \textit{branch-and-bound}~\cite{LandDoig1960, Cornuejols2006Optimization} search method, often perform better in these aspects, and are, therefore, among the most widely used tools for many hard MIP problems. 

Notably, commercially available branch-and-bound based solvers offered by CPLEX~\cite{cplex2009v12} and Gurobi~\cite{gurobi} have been widely adopted in the industry for optimization problems, such as traveling salesman, graph partitioning, and quadratic assignment problems~\cite{clausen1999branch}. These solvers employ the \textit{branch-and-cut} procedure, which effectively combines branch-and-bound with cutting-plane methods, and often reduces execution time. Despite their widespread success in solving practical problems of small and intermediate sizes, MIP problems are NP-hard in the worst case. Therefore, given the generic exponential runtime scaling and ever-increasing problem sizes for real-world PO problems, it is crucial to develop techniques that can improve the scope and scale of problems these solvers can handle.

\textbf{PO with near-term quantum computing.---} In parallel to the development of classical optimizers, quantum computing has shown promise for providing computational speedups for some optimization problems of relevance to science and industry~\cite{quant_speedup, abbas2023quantum, dalzell2023quantum}, particularly in portfolio optimization problems~\cite{herman2022survey, DAC24_review}. A variety of quantum algorithms have been proposed, including well-studied quantum heuristics like the quantum approximate optimization algorithm (QAOA)~\cite{Hogg2000,farhi2014qaoa}, which has been executed on quantum hardware for small-scale optimization problems~\cite{Harrigan2021,Pelofske2024,shaydulin2024evidence,niroula2022constrained}, including portfolio optimization~\cite{Buonaiuto2023, he2023alignment, herman2023constrained, sureshbabu2024parameter}, with up to approximately $100$ variables.

Another approach to quantum optimization is quantum annealing~\cite{Hauke_2020}, which solves optimization problems by finding the ground state of a classical Ising Hamiltonian, elevated to the quantum domain by describing a collection of interacting qubits. Like QAOA, quantum annealing has also been applied to portfolio optimization on quantum hardware~\cite{app122312288, cohen2008portfolio}. There have also been proposals for portfolio optimization based on quantum enhancements of classical algorithms. Specifically, quantum linear systems algorithms~\cite{harrow2009quantum, childs2012hamiltonian, chakraborty2018power} have been applied to solving reformulated constrained PO problems in both theory~\cite{rebentrost2018quantum} and hardware demonstrations~\cite{yalovetzky2024hybrid}. Quantum interior-point methods have also been proposed for this class of optimization problems~\cite{PRXQuantum.4.040325, kerenidis2019quantum}. Finally, although not specifically applied to portfolio optimization to date, quantum algorithms based on quantum walks~\cite{B_B_Ashley, chakrabarti2022universal} have been shown to improve branch-and-cut algorithms.

\begin{figure*}[!ht]
\includegraphics[width=\textwidth]{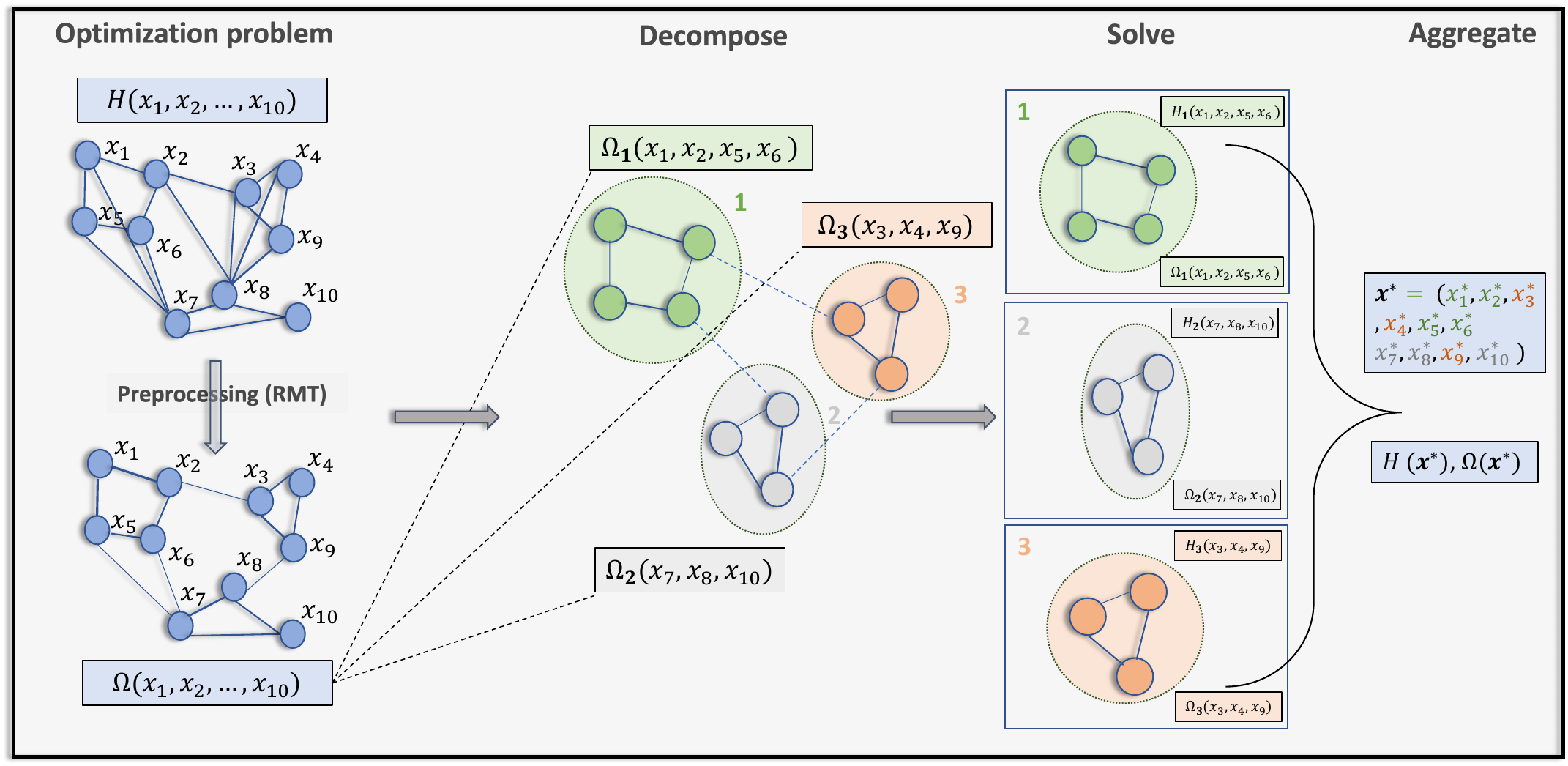}
\caption{Diagram of the decomposition pipeline to solve optimization problems, where the objective is denoted by $H$, and the constraint set is denoted by $\Omega$. Decompose: We first preprocess the correlation matrix using Random Matrix Theory (RMT). The resulting problem is then partitioned into a series of subproblems associated with the objectives $H_i$ and constraints $\Omega_{i}$, \( i={1,2,3} \). Solve: Each subproblem is solved (in parallel). Aggregate: The subproblems are aggregated and post-processed for feasibility checks into the final solution $\bm{x}^{*}$. Finally, we check the objective value $H(\bm{x}^{*})$, while ensuring feasibility of the solution vector $\Omega(\bm{x}^{*})$. Note that the aggregation technique shown is simple concatenation, but that one can generalized to more complex techniques, and that potentially different aggregated solutions could be tested.}
\label{fig:problem_figure}
\end{figure*}

A common challenge for implementing quantum algorithms on current hardware is the limited number of qubits, typically ranging from tens to hundreds on near-term devices, depending on the architecture and provider. Consequently, hardware demonstrations to date have been restricted to instances of portfolio optimization problems with a much smaller number of assets than problems solved in production, which can have thousands of variables.

Many financial assets exhibit a specific structure in which groups of assets are correlated within their communities, while being mutually anti-correlated with other communities~\cite{MacMahon_2015}. Moreover, it has been shown that this structure originates from the real market~\cite{block_po,chan1999portfolio}. These results suggest that there are opportunities for improving the performance of classical and quantum solvers by breaking down the problem into smaller subproblems based on such communities. However, to achieve this, it is not sufficient to simply apply standard community detection techniques that use correlation matrices instead of network data, as shown in Ref.~\cite{MacMahon_2015}. 

\textbf{Our contribution.---} In view of the aforementioned challenges for optimization faced by both classical and quantum computers, as well as the evidence of structure present in real-world PO problems, we propose a decomposition pipeline specifically designed to leverage this structure for PO problems, which can be generalized beyond portfolio optimization. Drawing from techniques used for community detection, our method consists of partitioning the original constrained optimization problem into subproblems of constrained optimization, which are each solved separately. The solutions to the subproblems are subsequently aggregated to output a solution to the original problem. 

The proposed decomposition pipeline consists of four steps, as shown in Fig.~\ref{fig:problem_figure}. First, we reinterpret the optimization problem by introducing a correlation \emph{graph} whose nodes correspond to assets, and whose edge weights correspond to the normalized covariance between returns of the assets. We apply preprocessing techniques based on random matrix theory to this graph in order to extract relevant signals. Second, we perform the decomposition by employing a modified spectral clustering based on Newman's method to partition the graph. The partition corresponds to segmenting the problem into smaller subproblems. Similarly, we decompose the constraint set of the original optimization problem, and we use risk rebalancing techniques before optimizing each subproblem. Third, we solve each of the subproblems individually. Finally, we aggregate the solutions of the subproblems to obtain an approximate solution to the main problem. 

Our proposed technique offers two main benefits. First, our decomposition pipeline improves the execution time of classical state-of-the-art solvers on portfolio optimization problems. We show that when utilizing the proposed decomposition pipeline and solving each subproblem with a state-of-the-art branch-and-bound based solver, the time-to-solution is significantly reduced (by at least $3\times$ for the largest problems considered in our numerical experiments, with $1500$ variables) as compared to directly solving the problem with the same solver. At the same time, for problems of sufficient size, even the decomposed subproblems become difficult to solve with classical techniques, but are are good targets for \textit{quantum} optimization. As such, the second benefit of our technique is that our pipeline contributes to the solution of portfolio optimization problems using quantum computers. Specifically, by leveraging the structure present in the data defining PO problems, we may potentially reduce the number of qubits required to implement the quantum optimization algorithm on quantum devices. We propose to utilize the structure present in typical problem instances to effectively reduce the problem size, thereby making large-scale problems compatible with near-term quantum hardware. Thus, our approach paves the way to near-term hardware demonstrations for practically-relevant applications at scale.

It is worth mentioning that decomposition techniques have been applied to a broad range of optimization problems~\cite{LEON20171334, rahmaniani2017benders}. A widely known partitioning procedure for solving MIP problems is Benders' decomposition (BD)~\cite{benders2005partitioning}, wherein the problem is decomposed into small subproblems containing only integer variables and other small linear programs containing only continuous variables. With continuous variables, graph theoretic decomposition algorithms have been explored in Refs.~\cite{port_cuts, arroyo2021dynamicportfoliocutsspectral}. In the context of near-term quantum computing, problem decomposition has been used to tackle graph clustering and graph partitioning problems on small quantum devices~\cite{shaydulincommunity,Shaydulin2019,Shaydulin20192,UshijimaMwesigwa2021}.

\section{Problem Statement and Motivation}
\label{prob_statement}

We focus on PO problems in the form of Mixed Integer Quadratically-Constrained Quadratic Programming (MIQCQP) problems. The generic form of MIQCQP is given by 
\begin{equation}
\begin{aligned}
    \min_{\bm{x}\in \{0,1\}^n} \quad & \bm{x}^T A_0 \bm{x} + 2\bm{b}_0^T\bm{x} + \kappa_{0} \\
    \text{subject to}          \quad & \bm{x}^T A_i \bm{x} + 2\bm{b}_i^T\bm{x} + \kappa_{i} \le 0, \quad \text{for }i=1,2,\ldots    
\end{aligned}\label{eqn:miqcqp}   
\end{equation}
We use the standard notation of bold lowercase letters, such as $\bm{x}$, for column vectors, uppercase letters, such as $A_i$, for matrices, and lowercase letters, such as $\kappa_{i}$, for scalars. In the above formulation, we write $\bm{x}$ as binary vectors, but it is straightforward to consider any vector $\bm{x}$ whose elements $x_i$ are integers within a predetermined range of values by, e.g., performing a binary expansion of $x_i$ and potentially adding new variables and constraints as needed. In the following, we consider two types of PO problems that can be cast as MIQCQPs. 

The first PO problem we consider is the canonical Markowitz mean-variance problem with constraints represented using integer variables. Given a universe of $n \in \mathbb{Z}$ assets to choose from, the objective is to maximize the return while minimizing the exposure to risk. We denote the vector of expected returns as $\bm{\mu} \in \mathbb{R}^{n}$, the vector of decision variables with discrete (binary or integer) elements as $\bm{x} \in \mathbb{Z}^{n}$, and the covariance matrix of returns as $\Sigma \in \mathbb{R}^{n \times n}$. For brevity, we will assume \textit{daily} returns without loss of generality, but it is straightforward to use other types of returns. The daily (rate of) return of an asset is defined as the difference between the closing price and the opening price, divided by the opening price. A standard formulation of this problem is defined by the objective function
\begin{equation}
     H(\bm{x}) := q~\bm{x}^T \Sigma \bm{x} -\bm{\mu}^T \bm{x}, \label{eqn:risk_aversed_po}     
\end{equation}
where the first term in the objective represents the exposure of risk, the second term represents the expected return, and their trade-off is governed by the risk aversion coefficient $q \in \mathbb{R^{+}}.$ The portfolio optimization task is to find a portfolio that minimizes the objective function, i.e., $\argmin_{\bm{x}\in \{0,1\}^n} H(x)$.

In practice, PO contains constraints, such as cardinality constraints on lower and upper limits of chosen assets. One can alternatively impose an equality constraint that fixes the exact number of chosen assets to be $\lceil dn\rceil$ for $0 \le d \le 1$, which can be expressed as $\mathbf{1}^T \bm{x} =dn$, where $\bm{1}$ is the all-ones vector.

In this work, we fix the number of assets with a linear equality constraint as 
\begin{equation}
\begin{aligned}
\min_{\bm{x}\in \{0,1\}^n} \quad &  q~\bm{x}^T \Sigma \bm{x}  -\bm{\mu}^T \bm{x}\\
\text{subject to} \quad &  \mathbf{1}^T \bm{x} =dn,  
\end{aligned}
\label{eqn:por_optim}
\end{equation}
where we omit the rounding of $dn$. The equality constraint $\bm{1}^T\bm{x} = dn$ is equal to two inequality constraints in Eq.~\eqref{eqn:miqcqp}; $\bm{1}^T\bm{x} - dn\le 0$ and $-\bm{1}^T\bm{x} +dn \le 0$. 

In the second PO problem we consider, one is given a (baseline) portfolio $\bm{x}_b$, and the goal is to identify a new portfolio $\bm{x}$ that is not far from $\bm{x}_b$, but which has a smaller overall risk exposure. Namely, the risk of changing $\bm{x}_b$ to $\bm{x}$ as measured by $\Sigma$ should be minimized. Typically, $\bm{x}_b$ represents a current portfolio that needs to be updated, or it could be the one based on a suggestion of subject matter experts. In this setting, the optimization problem can be formulated as
\begin{equation}
\begin{aligned}
& \underset{\bm{x}\in \mathbb{Z}^n}{\text{min}}
& & (\bm{x}-\bm{x}_b)^T \Sigma (\bm{x}-\bm{x}_b) \\
& \text{subject~to}
& & \bm{x}^T \Sigma \bm{x}~\leq~a,
\end{aligned}
\label{eqn:quad_constraint_ps}
\end{equation}
where $a$ is a scalar strictly less than $ \bm{x}_b^T \Sigma \bm{x}_b$ so that the optimal solution, if it exists, is different from $\bm{x}_b$. Notice that $\bm{x}$ is a vector of integers instead of a vector of binaries, but it is straightforward to derive a corresponding MIQCQP as shown in Appendix~\ref{subsec:app_ip2miqcqp}. 

A popular method to compute the exact solution to the cardinality-constrained PO \eqref{eqn:por_optim} is the branch-and-bound method~\cite{LandDoig1960}, which is a recursive method that explores solutions by first \textit{branching} into the subspace of the solutions of smaller subproblems, then obtaining good upper and lower \textit{bounds} that are used to skip searches on a subset of branches. For example, in Eq.~\eqref{eqn:por_optim} the branching can be done by splitting the problem into subproblems with $x_n$ fixed to either $0$ or $1$, i.e., $H_0(\bm{x})$ with $x_n=0$ and $H_1(\bm{x})$ with $x_n=1$. Each of the resulting subproblems has one less variable and can also be recursively branched. During the exploration of $H_0(\bm{x})$, one can obtain a feasible solution (upper bound), and an estimated optimal solution (lower bound) that may allow to skip exploring the branches of $H_1(\bm{x})$. The \textit{integrality gap} is defined as the difference between the lower bound, which is achievable by relaxation, and the upper bound, which is realized by integer variables. Notice that the branching can be done in arbitrary order of the subsets of decision variables, which can greatly influence the time to find optimal solutions (or, solutions with the lowest integrality gaps). 

As mentioned in the introduction, the correlations between financial assets give rise to a clear structure. We provide empirical evidence of this underlying structure originating from the real market~\cite{block_po,chan1999portfolio} that the solver leverages to find the solution more efficiently. To illustrate this, we solve the PO problem introduced in Eq.~\eqref{eqn:por_optim} with a linear constraint, constructed using both (i) market data from the S\&P~500, and (ii) randomly-generated covariance matrices (drawn from the Wishart distribution~\cite{Wishart1928}), which we discuss in Appendix \ref{appendix:numerics}, and where we compare the time-to-solution (TTS) of solving these two cases.

There are some PO problems that are especially hard empirically, potentially due to complexity in the objective or constraints. As we will show, in Sec.~\ref{sec:numerical_results} for example, the risk minimization with quadratic constraint (Eq.~\eqref{eqn:quad_constraint_ps}) formulated with market data (S\&P 500) presents a TTS higher than the simpler problem a with linear constraint presented before. However, the distinction between random covariance and S\&P 500 data is less pronounced, suggesting that classical solvers may struggle to utilize the underlying structure in problems with more complex constraints, further motivating the need for our decomposition framework.

\section{Decomposition Pipeline}
\label{sec:decomposition}

In this section we describe the proposed decomposition pipeline that leverages the structure of the covariance matrix for typical portfolio optimization problems. Our approach consists of three steps presented with details in the next subsections: preprocessing of the correlation matrix, clustering that utilizes a correlation-based modularity metric, and the partitioning of the constraints, as shown in Fig.~\ref{fig:problem_figure}. As expected, the quality of the solution from the decomposed pipeline depends on the number of clusters and the structure of the decomposed covariance matrices. We give an analytical bound on the expected decrease of the solution of the decomposed pipeline in Appendix~\ref{app:analyticsl_1}. 

\subsection{Pre-Processing of Correlation Matrices}
\label{sec:pre_processing}

In general, given time-series data, as is often the case for the returns of financial assets, the correlation matrices are constructed in the following way. Given the daily returns of $n$ different assets over a period of $\mathcal{T}$ days, we define the matrix $X_{n \times \mathcal{T}} \in \mathbb{R}^{n \times  \mathcal{T} }$, where the $i$-th row corresponding to the $i$-th asset is a vector containing the daily returns of that asset across the period of $\mathcal{T}$ days. 

In the setting of finite $n$ and $\mathcal{T} \rightarrow \infty$, given the $X_{n \times \mathcal{T}}$, the usual sample covariance estimator is $\Sigma = (1/\mathcal{T}) XX^T,$ and it converges almost surely to the true covariance $\overline{\Sigma}$, making it a strongly consistent estimator of the population covariance matrix \cite{Bun_2017}. In particular, the expected value of the Frobenius norm of the difference $\| \Sigma - \overline{\Sigma} \|_F^2$, i.e., the mean squared error (MSE), tends to zero in this limit. In practice, $\mathcal{T}$ is finite. One way of improving the maximum likelihood estimator is with shrinkage, as originally observed in Refs.~\cite{james1961estimation, Stein1986}. Another challenge is that in many practical settings, the maximum likelihood estimator is not the right choice because it comes with considerable noise in estimation. 

When faced with the challenge of estimating covariance matrices in such regimes, two strategies come to the fore: thresholding and shrinkage estimators. Thresholding serves to simplify the covariance matrix by setting small correlations, those considered to be noise, to zero, which can be beneficial in terms of computational efficiency and interpretability, but can discard potentially valuable information. On the other hand, shrinkage estimators \cite{shrinkage_LEDOIT,  LW2012, LW2015, LW2017b, LW2020} aim to improve the estimation of a covariance matrix by combining the sample covariance matrix with a structured estimator, such as the identity matrix. 

Once $\Sigma$ is estimated, preprocessing techniques are applied. Given an estimated $\Sigma$, the correlation matrix $C$ can be written directly as 
\begin{equation}
C = D^{-1/2} \Sigma D^{-1/2},
\end{equation}
where \( D \) is a diagonal matrix \( D \equiv \mbox{Diag}(\Sigma_{11}, \Sigma_{22}, \ldots, \Sigma_{nn}) \). Thus, each element of $C$ is given by 
\begin{equation}
C_{ij} = \frac{\Sigma_{ij}}{\sqrt{\Sigma_{ii} \cdot \Sigma_{jj}}}.
\end{equation}  
As shown in Fig.~\ref{fig:RMT_split}, for a correlation matrix \(C\) derived from \(n\) completely random time series of duration \( \mathcal{T}\), a component of the eigenvalue distribution follows the Marchenko-Pastur (MP) distribution \cite{1967SbMat...1..457M}, as \(n\) and \(\mathcal{T}\) are both large but with a fixed ratio $\beta=n/\mathcal{T}$. The density function of this distribution is given by,
$$
\rho(\lambda) =
\begin{cases} 
\frac{\sqrt{(\lambda_{+} - \lambda)(\lambda - \lambda_{-})}}{2\pi\lambda\beta\sigma^2}, & \text{if } \quad \lambda_{-} \le \lambda \le \lambda_{+}\\
0, & \text{otherwise},
\end{cases}
$$
where \(\lambda_{\pm} = \sigma^2\left[1 \pm \sqrt{\beta}\right]^2\) represent the maximum and minimum eigenvalues respectively. For true correlation matrices, $\sigma=1$. However, in practice, we can treat $\sigma$ to be an adjustable parameter \cite{bouchaud2009financial, BUN20171}. This is because several eigenvalues remain above $\lambda_{+}$ and carry some information, thereby decreasing the variance of the effectively random component of the correlation matrix. Similarly, $\beta$ should be treated as an adjustable parameter \cite{Bouchaud_Potters_2003} because effectively correlated samples are redundant and the sample correlation matrix should behave as if we had observed not $\mathcal{T}$ samples but an effective number $\mathcal{T}^{*}$ or written directly $\beta^{*}=(n/\mathcal{T}^{*})$. After the appropriate fits, we obtain $\lambda_{+}$ and $\lambda_{-}$ as shown in Fig.~\ref{fig:RMT_split}. The utility of the Marchenko-Pastur distribution lies in its ability to characterize those eigenvalues that are predominantly noise-induced. Specifically, eigenvalues within the interval \([\lambda_-, \lambda_+]\) are attributed to random fluctuations, whereas those exceeding \(\lambda_+\) signify significant underlying structures, often reflecting meaningful relationships in the data \cite{Bouchaud_Potters_2003}. These relevant eigenvalues are highlighted with the blue and red stars in Fig.~\ref{fig:RMT_split}. 
\begin{figure}[h!]
    \includegraphics[width=9cm]
    {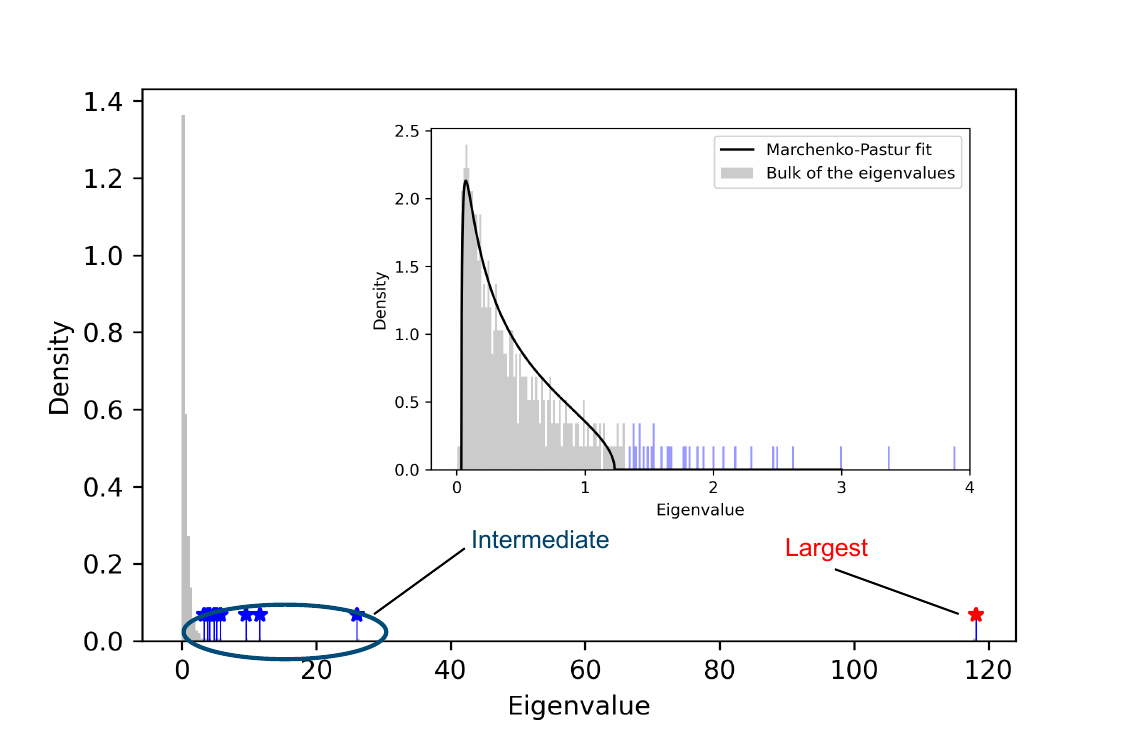}
    \caption{Eigenvalue density of a correlation matrix derived from S\&P~500 market data, with $n=484$ stocks and returns data collected from 2014 to 2019 for $\mathcal{T}=1260$, followed by data cleaning. The bulk of the eigenvalues form a dense cluster on the left side of the figure, consistent with the hypothesis of a random correlation matrix, followed by intermediate eigenvalues (shown in blue) and the largest eigenvalue (red star). In the inset, the bulk region is fitted with a Marchenko-Pastur distribution, represented by a smooth black line. The black line fit \cite{Laurent_RMT} has a lower edge at $\lambda_{-}=0.09$ and upper edge at $\lambda_{+}=1.23$. The upper edge of the Marchenko-Pastur density can serve as a threshold to identify the noisy part of the correlation matrix.} 
    \label{fig:RMT_split}
\end{figure}

To apply this insight to empirical correlation matrices, we decompose the matrix \(C\) into $C = C_{\mbox{noise}} + C_{\mbox{signal}}$, where
\begin{equation}
C_{\mbox{noise}} = \sum_{i: \lambda_i \leq \lambda_+} \lambda_i \bm{v}_i\bm{v}^T_i,
\end{equation}
\begin{equation}
C_{\mbox{signal}} = \sum_{i: \lambda_i \geq \lambda_+} \lambda_i \bm{v}_i\bm{v}^T_i,
\end{equation}
the eigendecomposition of $C$ whose eigenvalues are at most $\lambda_{+}$ captures the random components associated with eigenvalues less than or equal to \(\lambda_+\), and \(C_{\mbox{signal}}\) represents the structured components reflecting eigenvalues greater than \(\lambda_+\). The largest eigenvalue, $\lambda_1$, leads to top eigenvalue bias in the study of spiked covariance models~\cite{Bouchaud_Potters_2003, Donoho2018Optimal}. $\lambda_1$ is typically much greater than \(\lambda_+\), often corresponding to the global or \textit{market mode}~\cite{Bouchaud_Potters_2003, MacMahon_2015}, a common factor affecting all entities within the dataset. To discern more nuanced relationships, it is essential to further decompose the signal as,
\begin{equation}
C_{\mbox{signal}} = C^{*} + C_{\mbox{global}}, 
\label{eqn:3_way_split}
\end{equation}
where \(C_{\mbox{global}} = \lambda_1 \bm{v}_{1}\bm{v}^T_{1} \) represents the market mode, and the relevant structure present in the data is represented by $C^{*}$, which captures correlations within sub-groups of stocks or variables that are influenced similarly by shared factors. Specifically,
\begin{equation*}
C^* = \sum_{\substack{i: \lambda_i \geq \lambda_+ \\ i \neq 1}} \lambda_i \bm{v}_i\bm{v}^T_i.
\end{equation*}
We leverage $C^{*}$ in order to identify communities of assets to define the subproblems as we will present in the next subsections.  

\subsection{Clustering with Correlation-Based Modularity}
\label{sec:Clustering_with_Correlation_based_Modularity}

Our goal is to leverage the structure in the correlation matrix to identify subsets of assets defining the subproblems. To this end, we create the \textit{correlation graph}, $G$, whose nodes correspond to assets, and whose edges are weighted by the total correlation between their corresponding nodes (assets in $C$). We then perform clustering to identify communities. Standard community detection algorithms for networks use a modularity metric usually defined in terms of the difference between the actual edge weight \( C_{ij} \) and the expected edge weight under some null model. The expected edge weight is typically estimated based on the product of the degrees \( k_i \) and \( k_j \) of the nodes \( i \) and \( j \). Thus, the naive correlation based modularity \( Q \) under the community grouping $\bm{g}$ can be expressed as, 
\begin{equation}
Q(\bm{g}) = \frac{1}{\gamma} \sum_{ij} \left[ C_{ij} - \frac{k_i k_j}{\gamma} \right] \delta(g_i, g_j),
\label{eqn:original_modularity}
\end{equation}
where \( C_{ij} \) represents the actual edge weight between nodes \( i \) and \( j \), \( \gamma \) is the sum of all of the edge weights in the network, and $g_i$ is the component of $\bm{g}$, which denotes the group assignment for the $i$-th node. \( \delta(g_i, g_j) \) is the Kronecker delta function equal to 1 if nodes \( i \) and \( j \) are in the same community and 0 otherwise. However such standard community detection algorithms, including the ones developed by Newman~\cite{Newman_2006}, do not consider the specific properties of the modularity matrix when derived from correlation matrices \cite{MacMahon_2015}. The authors also highlight the need for a revised null model that acknowledges the structure of correlations in the network. A suitable null model for networks represented by correlation matrices must be formulated to account for both the individual properties of the nodes and the overarching network structure. Such a model ensures that the modularity function accurately reflects the network's community division. To this end, we outline  modifications to spectral modularity maximization methods for community detection algorithms in 
Algorithm~\ref{algorithm_modified_spectral}.  

We utilize tools from random matrix theory as described in Sec.~\ref{sec:pre_processing} to define a modified correlation matrix $C^{*}$ that captures the relevant information. We then use the correlation matrix $C^{*}$ to define the edge weight between nodes in the correlation graph. Recall that in Eq.~\eqref{eqn:original_modularity} the correlation-based modularity is computed with regards to multiple groups or communities. In this paper, however, we focus on the recursive greedy approach of partitioning into two communities similar to Ref.~\cite{Newman_2006}. For this case, let us rewrite $\bm{g}$ as $\bm{z}$ and let \( z_i = +1 \) if node \( i \) belongs to a particular community, and \( z_i = -1 \) otherwise. Then, $\delta(g_i, g_j) \equiv (z_i z_j + 1)/2$, and a modified correlation-based modularity \( Q_c(\bm{z}) \) can be defined as:
\begin{equation}
    Q_c(\bm{z}) = \frac{1}{ \gamma } \sum_{i,j} \frac{C_{ij}^{*} (z_i z_j + 1)}{2}
         =\frac{\mathbf{z}^T C^* \mathbf{z}}{2 \gamma } + \frac{C^*_{tot}}{2 \gamma },\label{eqn:C_tot}
\end{equation}
where $C^*_{tot}$ denotes the total edge weight of the cleaned correlation matrix $C^*$. By relaxing $\bm{z}$ variables to be real (while satisfying $\bm{z}^T\bm{z}=n$~\cite{Newman_2006}), the vector $\mathbf{z}$ maximizing $Q_c(\mathbf{z})$ is easily seen to be the vector matching the signs of the components of the eigenvector of $C^*$ corresponding to the largest eigenvalue. Additionally, the second term in \ref{eqn:C_tot} can be dropped from the optimization objective because it does not depend on $\mathbf{z}$. In Theorem~\ref{thm:solution_gap}, we show that choosing the eigenvector corresponding to the largest eigenvalue would indeed be optimal for minimizing the drop in solution quality in the setting where we use continuous relaxation of variables.

To further subdivide the communities obtained from the initial bisection, we need to determine the potential modularity change \(\Delta Q_c\). The graph \(G\) is then split into two new subgraphs (\(G_1\) and \(G_2\)). Let \(\bm{z}\) be a vector representing the vertices in \(G\). The modularity change for this bisection is denoted by $\Delta Q_c^{(G_1|G_2)}$, and is given by
\begin{eqnarray}
\Delta Q_c^{(G_1|G_2)} &=& \frac{1}{\gamma} \left[ \sum_{i,j \in G_1} C_{ij}^{*} + \sum_{i,j \in G_2} C_{ij}^{*} - \sum_{i,j \in G} C_{ij}^{*} \right] \nonumber \\
&=& \frac{1}{\gamma} \left[ \sum_{i,j \in G} C_{ij}^{*} \frac{1+z_i z_j}{2} - \sum_{i,j \in G} C_{ij}^{*} \right] \nonumber \\
&=& \frac{1}{2\gamma} \left[ \sum_{i,j \in G} C_{ij}^{*} z_i z_j - \sum_{i,j \in G} C_{ij}^{*} \right] \nonumber \\
&=& \frac{ \bm{z}^{T} C_G^{*}\bm{z} }{2\gamma} - \frac{\widetilde{C}_{GG}^{*}}{2\gamma},
\label{eqn:Qc_gain}
\end{eqnarray}
where \(C_G^{*}\) represents the sub-matrix of \(C^{*}\) restricted to the subset of nodes within graph \(G\), and  $\widetilde{C}_{GG}^{*}:=\sum_{i,j \in G} C_{ij}^{*}$. As with the initial bisection, \(\bm{z}\) is chosen to maximize \(\Delta Q_c^{(G_1|G_2)}\) by selecting its elements to match the sign of the eigenvector corresponding to the largest eigenvalue of the matrix \(\mathbf{C}_G^{*}\) \cite{spectral_mod, MacMahon_2015}. Then, the proposed modularity-based spectral method continues iterating this procedure until no further bisection can increase the modularity. We summarize the steps in Algorithm~\ref{algorithm_modified_spectral} and Algorithm~\ref{algorithm_valid_split} (see Appendix~\ref{app:alg_2}).

The proposed method differs from standard spectral clustering in that it specifically aims to maximize a modularity function that measures the strength of division of a network into modules or communities. This is achieved by eigen-decomposing the modularity matrix, a matrix that represents the difference between the actual adjacency matrix of the network and a null model. The method involves recursively or iteratively bisecting the network based on the leading eigenvector of the modularity matrix, contrasting with spectral clustering, which typically partitions a network based on the eigenvectors of the Laplacian matrix and does not necessarily aim to maximize modularity.

Another alternative for a modified modularity metric is to utilize a \textit{normalized} \(C^{*}\) instead of \(C\), which is defined as,  
\begin{equation}    
(\widetilde{C^{*}})_{ij} \leftarrow C^{*}_{ij} - \frac{k_i k_j}{\gamma},
\label{eqn:new_modularity}
\end{equation}    
where $k_{i}=\sum_{j}C^{*}_{ij}$. Note that $\bm{1}=(1,1,\ldots)^{T}$ is a trivial eigenvector of the normalized $\widetilde{C^{*}}$ with eigenvalue $0$, because the row-sums and the column-sums of $\widetilde{C^{*}}$ is $0$. This is reminiscent of a property of the matrix known as the graph Laplacian \cite{chung1997spectral}. Hence, the optimal eigenvector of $\widetilde{C^{*}}$ is orthogonal to the all $\bm{1}$ vector~\cite{Newman_2006}, thus satisfying $\sum_{i}z_{i}=0$. This ensures that we can find a bipartition corresponding to the positive and negative components of the optimal eigenvector of the modified modularity matrix. We can use this observation further to perform the clustering process in order to find all communities smaller than a given threshold. We introduce Algorithm~\ref{algorithm_modified_iterative} (which is a slight modification of Algorithm~\ref{algorithm_modified_spectral}, see Appendix~\ref{app:alg_2}) that iteratively decomposes the problem further until the number of variables in each problem is less than a certain threshold. Specifically, at each step, if a community size is over a threshold, we re-run the community detection algorithm as many times needed until all the subcommunities generated are within the expected size. Note that we can improve this algorithm by regrouping small communities together, which should reduce the solution error to the original problem.

\algrenewcommand\algorithmicrequire{\textbf{Input:}}
\algrenewcommand\algorithmicensure{\textbf{Output:}}

\begin{algorithm*}[tbh]
\begin{algorithmic}[1]
\Require $C^{*}$: Clean correlation matrix as in Eq.~\eqref{eqn:3_way_split}
\Ensure $\bm{g}$: vector of community assignments 
\State Initialize: $G \gets C^{*}$, $U \gets G$, $D \gets \emptyset$  \Comment{$G$: Input graph , $U$: Undecided subgraphs }
\While{$U \neq \emptyset$}
    \State $G' \gets \text{pop}(U) $: \Comment{Details in Sec.~\ref{sec:Clustering_with_Correlation_based_Modularity}}
    \State $(\text{isValid}, G_1, G_2) \gets \text{VALID\_SPLIT}(G')$ \Comment{Check if $G'$ can be split (details in Algorithm~\ref{algorithm_valid_split})}
    \If{$\text{isValid}$}
        \State $U \gets U \cup \{G_1, G_2\}$ \Comment{If $G'$ can be split, store them to be split again}
    \Else
        \State $D \gets D \cup \{G'\}$ \Comment{Otherwise, fix the community assignment}
        \State Update $\bm{g}$ to reflect the community labels for the vertices in $G'$
    \EndIf
\EndWhile
\State \Return $\bm{g}$ as computed from the sets in $D$
\end{algorithmic}
\caption{Modified Modularity-Based Spectral Optimization}\label{algorithm_modified_spectral}
\end{algorithm*}
 
When executing our algorithms, we need to ensure that each node is assigned to a community, or that each node belongs to a community whose size is at most the maximum allowed size for a community (denoted by $\phi$) after executing Algorithm~\ref{algorithm_modified_spectral} or Algorithm~\ref{algorithm_modified_iterative}, respectively. For this, we introduce the \textit{pop} method to choose a subgraph whose nodes have not been assigned to a community in Algorithm~\ref{algorithm_modified_spectral} or to choose a subgraph whose size exceeds the threshold in Algorithm~\ref{algorithm_modified_iterative}. We denote the set of subgraphs that have not yet been assigned to a community by $U$. There are several possible strategies to pick a subgraph from $U$: (i) pick the first element in the First-In-First-Out order, (ii) pick a subgraph uniformly at random, (iii) choose it with probability proportional to its size, or (iv) prioritize the largest subgraph. In practice, we observe that approaches (i) and (iii) outperform the other choices. 

\subsection{Partitioning of Constraints \label{sub:partition}}

In the previous subsection, we have shown how to divide the initial asset universe of the problem into subsets of assets to define subproblems. To do this, it is also required to partition the constraints of the initial problem. We show how to do this for two types of constraints in next subsections. 

\subsubsection{Cardinality Constraints}
\label{card_cons}

Proceeding with the earlier formulation of dividing the asset universe into \( K \) communities, we denote the communities by the index $k$. Each community has $n_k$ assets, expected returns $\bm{\mu}_k$, and covariance matrix $\Sigma_k$. Within each community \( k \), the subproblem is formulated as, 
\begin{equation}
\min_{\bm{x}_k\in \mathbb{Z}^{{n}_{k}}} H_k(\bm{x}): -\bm{\mu}_k^T \bm{x}_k + q' \bm{x}_k^T \Sigma_k \bm{x}_k, \label{eqn:opt_commk}
\end{equation}
where \( \bm{x}_k \) denotes the vector of asset weights in community $k$ with $n_k$ assets. Each community adheres to a local cardinality constraint $\Omega_{k}: \mathbf{1}^T \bm{x}_k = dn_k$. After solving these subproblems, the solutions \( \bm{x}^{*}_1, \bm{x}^{*}_2, \ldots, \bm{x}^{*}_K \) are recombined to respect the global cardinality constraint: $\Omega:\sum_{k=1}^K \mathbf{1}^T \bm{x}^{*}_k = dn$. The validity of this combined solution is determined by its adherence to both local community constraints and the global constraint. In order to satisfy the global constraint $\Omega$, we can apply rounding. For the first $K-1$ communities (i.e., $k=1,2,\dots, K-1$), the number of selected assets is the one output by each separate subproblem. Then, the number of assets for the $K$-th community is enforced to be $\lfloor dn \rfloor - \sum_{k=1}^{K-1} \lfloor dn_k \rfloor.$ 

Similarly, we restrict the risk of the subproblems relative to the total exposure of risk of the initial problem. To this end, we perform risk rebalancing, which adjusts the portfolio to maintain a desired level of risk exposure. We modify the objective function in Eq.~\eqref{eqn:opt_commk} by replacing \( q \) with \( q' \) as,  
\begin{equation}
q' = q\times\left(\frac{ \sum_k \|  \bm{\mu}_{k} \|_2 / \| \bm{\mu} \|_2}{\sum_k \|  \Sigma_k \|_F / \| \Sigma \|_F}\right).
\end{equation}
To motivate the need for rebalancing, observe that the risk factor $q$ is sensitive to the units in which returns and variances are measured. For example, if we scale all prices by a constant factor $\alpha$, then the risk factor $q$ must be scaled by $1/\alpha$ to retain the same solution. For a given risk tolerance, therefore, the corresponding $q$ is dependent on the norms of $\Sigma$ and $\bm{\mu}$. After community detection, we are effectively solving portfolio optimization with a block diagonal covariance matrix with blocks given by $\Sigma_k$. The proposed rebalancing attempts to preserve the risk tolerance under this modification. We cannot analytically guarantee this, because the subproblems are MIQPs and are subject to rounding. However, this rule is empirically observed to be very effective in minimizing the loss in solution quality compared to the original problem.

\subsubsection{Quadratic Constraints}
\label{quad_cons}

Consider the particular PO problem with a quadratic constraint introduced in Eq.~\eqref{eqn:quad_constraint_ps}. In a decomposed optimization framework, the overall problem is partitioned into a series of subproblems. Each subproblem \( k \) is associated with an objective involving a quadratic term,
$$
(\bm{x}_k - (\bm{x}_b)_{k})^T \Sigma_k (\bm{x}_k - (\bm{x}_b)_{k}),
$$
where \( \bm{x}_k \) represents the decision variable for the \( k \)-th subproblem, and \( \Sigma_k \) is a symmetric positive definite matrix associated with the \( k \)-th subproblem. The subproblems are constrained by $\bm{x}_k^T \Sigma_k \bm{x}_k \leq w_k a,$ where \( w_k \) represents the weight on the \( k \)-th subproblem,  \((\bm{x}_b)_{\text{k}}\) represents the elements of $\bm{x}_b$ vector corresponding to the community $k$, and \( a \) is a scalar parameter shared across all subproblems. One simple choice for $w_{k}$ can be to use the ratio of nodes present in the \( k \)-th subproblem over the total number of assets.

Satisfying these constraints locally does not guarantee that the global constraint is satisfied. To enforce the feasibility of the global constraint, we introduce a new matrix \( \Sigma'' = s \Sigma' \), where \( \Sigma'\) is a block diagonal matrix constructed from the matrices \( \Sigma_k \), and \( s \) is a scalar. This matrix \( \Sigma'' \) is designed to be diagonally dominant over the original covariance matrix \( \Sigma \), thereby ensuring that \( \Sigma'' - \Sigma \succeq 0 \). To ensure this, we need to derive conditions on the scalar \( s \). The criterion is derived based on the relationship
\begin{equation}
    s \geq \frac{\lambda_{\text{max}}(\Sigma)}{\lambda_{\text{min}}(\Sigma')} \ge 1,
\end{equation}
where \(\lambda_{\text{max}}(\cdot)\) and \(\lambda_{\text{min}}(\cdot)\) represent the maximum and the minimum eigenvalues, respectively. The minimum eigenvalue of $\Sigma'$ is the minimum among all the block diagonal constituents, because their union forms the eigenset of $\Sigma'$. This condition ensures that the scaling factor \(s\) sufficiently elevates the eigenvalue spectrum of \(\Sigma'\) to meet or exceed that of \(\Sigma\). With the modified constraint for each subproblem being
$$
    \bm{x}_k^T \Sigma_k \bm{x}_k \leq \frac{w_k a}{s},
$$
and assuming the normalization condition $\sum_k w_k = 1$ is satisfied for feasible solutions, the aggregate of these constraints across all subproblems implies
$$
\bm{x}^T \Sigma' \bm{x} = \sum_k \bm{x}_k^T \Sigma_k \bm{x}_k  \le \frac{a}{s} \sum_k w_k \le a. 
$$
Alternatively, the initial step in this strategy involves setting the weight of each subproblem's constraint, \(w_k a\), to be proportional to a predefined quadratic form: $w_k a = r \cdot (\bm{x}_b)_{\text{k}}^T \Sigma_{\text{k}} (\bm{x}_b)_{\text{k}},$ where $r < 1$ is a constant. The choice of $r$ is explained in Appendix~\ref{sec:deriving_r} so that
\begin{equation}
a = r \sum_k (\bm{x}_b)_{\text{k}}^T \Sigma_{\text{k}} (\bm{x}_b)_{\text{k}} < \bm{x}^T_b \Sigma \bm{x}_b
\end{equation}
eliminates the trivial solution $\bm{x}_b$ in Eq.~\eqref{eqn:quad_constraint_ps}. The suppression factor \( s \) for feasibility of the constraints can be further introduced modifying the initial weighting to $w_k a/s = (r/s) \cdot (\bm{x}_b)_{\text{k}}^T \Sigma_{\text{k}} (\bm{x}_b)_{\text{k}},$ where increasing \(s\) serves to systematically tighten the constraints. The choice of \(s\) is critical: it must be large enough to ensure the feasibility of all subproblems while avoiding overly restrictive constraints that might preclude feasible solutions. In practice, we do a simple search over a range of values starting from $1$. 

\section{Numerical Results}
\label{sec:numerical_results}

We now numerically assess the performance of our decomposition pipeline in its application to PO problems based on real-world data. 

\textbf{Problem instances and figures of merit.---} The dataset utilized in this work is based on the Russell 3000 index, with daily returns collected over a period of 1000 days starting from January 2010. The instances range in size from $90$ to $1500$ assets. Each random seed corresponds to adding incrementally a random pool of assets in step sizes. Branch and Bound (B\&B) solvers are characterized by the MIP gap (default $10^{-4}$,) which corresponds to the difference between the upper and lower bounds on the solution:
\[
\text{MIP gap} = \frac{|H_{B} - H_{I}|}{|H_{I}|}.
\]
Here, \( H_{I} \) is the best known integral bound on the objective value (i.e., the best integer feasible solution found so far, also referred to as incumbent), and \( H_{B} \) is the best known upper bound, being this the optimal value corresponding to the relaxed problem in the B\&B procedure. The solver terminates successfully as soon as it has proven that the gap of the best found solution is below a specified small value; we refer to this time as time-to-solution (TTS). 

This time is one figure of merit, together with the solution quality $H$. The problems considered for the benchmark are the two PO problems introduced in Sec.~\ref{prob_statement}. We benchmark the performance of the proposed decomposition, where we solve each subproblem utilizing a B\&B solver. We compare to the results obtained through \textit{direct optimization} which refers to the optimization of the full problem. The \textit{direct optimization} is done with both the default MIP gap ($10^{-4}$), and a higher MIP gap ($5 \times 10^{-2}$) set with respect to the loss in solution quality. For the decomposition pipeline, we benchmark the performance without restricting the community sizes, part of the decomposition component of the pipeline, and by restricting the size of the communities up to $30$. The motivation of this number is to make these subproblems compatible with near-term quantum devices that are characterized by a limited number of qubits (usually in the order of tens or hundreds). Given that one generally needs a number of ancilla qubits to faithfully embed the given problem onto the underlying quantum hardware, we limit the size of the subproblems (i.e., communities) to $30$.

\begin{figure}[!h]
    \begin{overpic}[percent, scale=0.55]{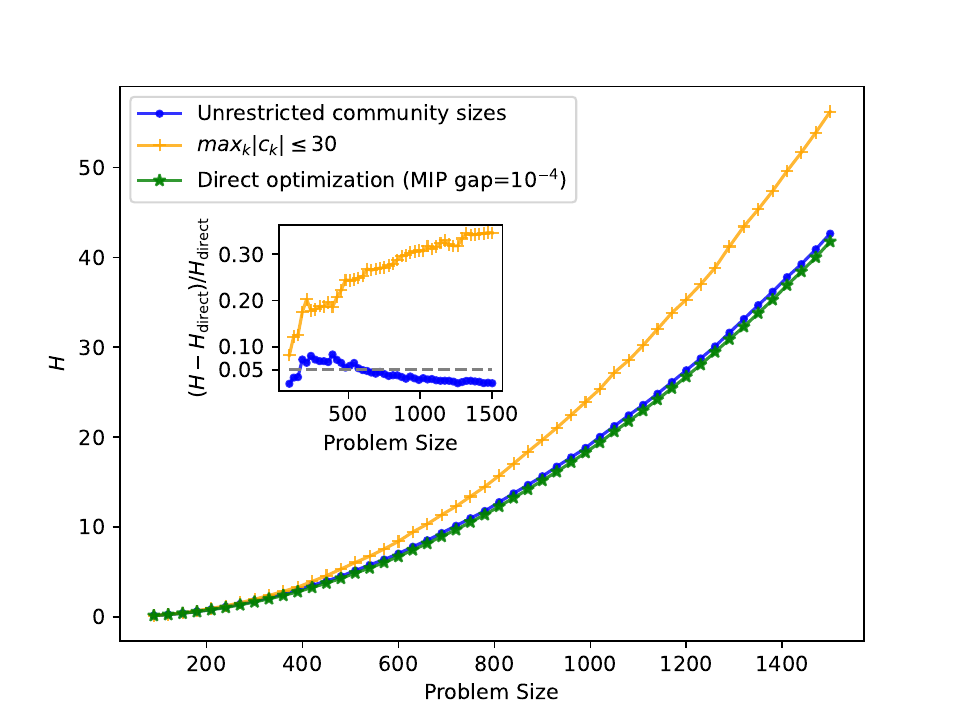}
    \put(5,68){(a)}
    \end{overpic}
    \medskip
    \begin{overpic}[percent, scale=0.55]{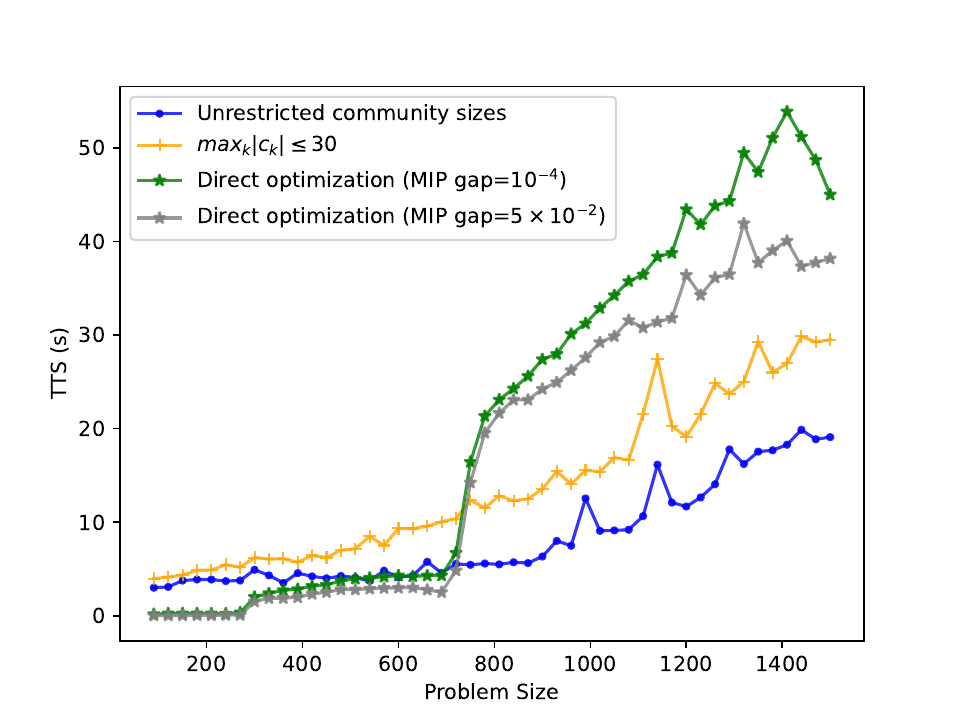}
    \put(5,68){(b)}
    \end{overpic}
    \caption{Performance as a function of the problem size of the PO with linear cardinality constraint set at $n/2$ (Eq.~\eqref{eqn:por_optim}). Panel (a) (top) shows the solution quality $H$ and panel (b) (bottom) the time to solution (TTS). The line (stars) corresponds to directly optimizing with the green line corresponding to the upper bound of the MIP gap set at $5 \times 10^{-4}$ and the grey line with the bound at $5 \times 10^{-2}$. The blue line (dots) and the orange (plus signs) refer to the proposed decomposition pipeline, where the first one corresponds to an unrestricted community size and former to a restricted size of the largest community size thresholded at $30$. The inset of panel (a) shows the relative drop in objective for our decomposition pipeline with and without restrictions over the community size. The dashed (gray) line marks the drop or gap percentage at $0.05$ indicating that indeed, the loss in solution quality as a consequence of the approximations imposed by the pipeline is bounded due to the structure present in the data.}\label{fig:lin_combined}
\end{figure}

\begin{figure}[!hb]
    \centering
    \includegraphics[width=\columnwidth]{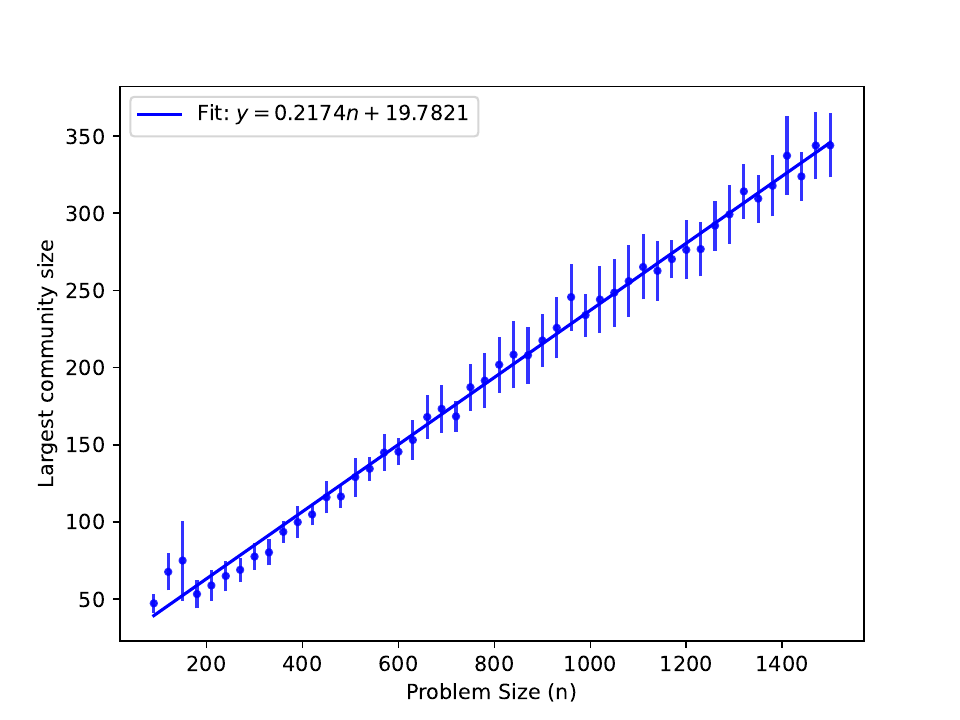}
    \caption{Size of the largest community against problem size ($n$). The error bars correspond to the standard deviation over $10$ seeds for each problem size. The median is fit with a linear curve with slope $\sim 0.21$.}
    \label{fig:size_largest}
\end{figure}

\begin{figure}[!h]
    \begin{overpic}[percent, scale=0.55]{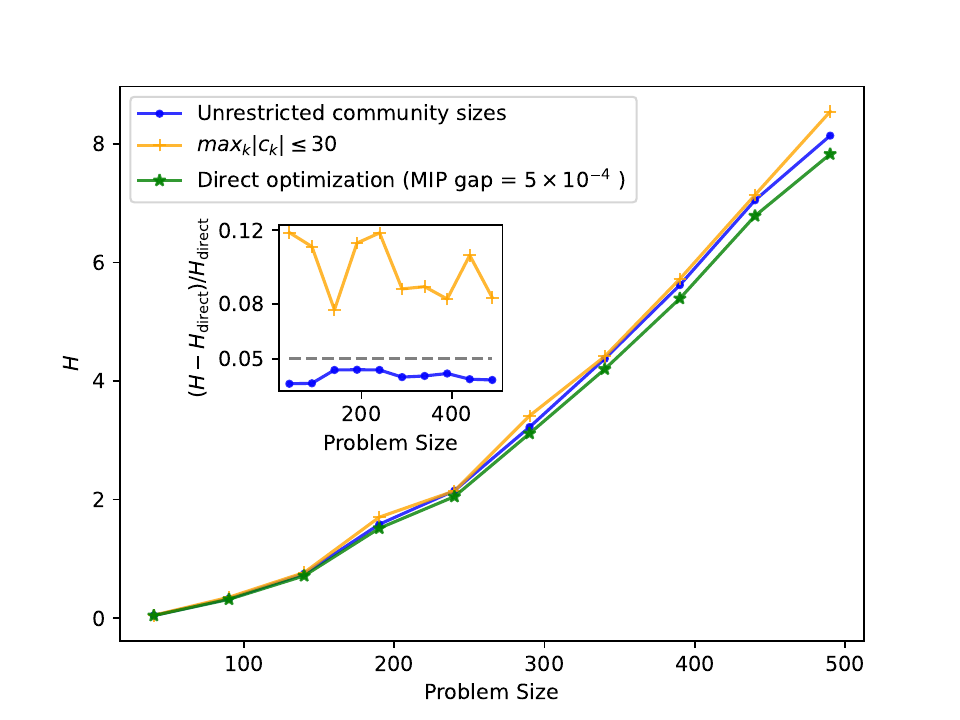}
    \put(5,68){(a)}
    \end{overpic}
    \medskip
            \centering
     \begin{overpic}[percent, scale=0.55]{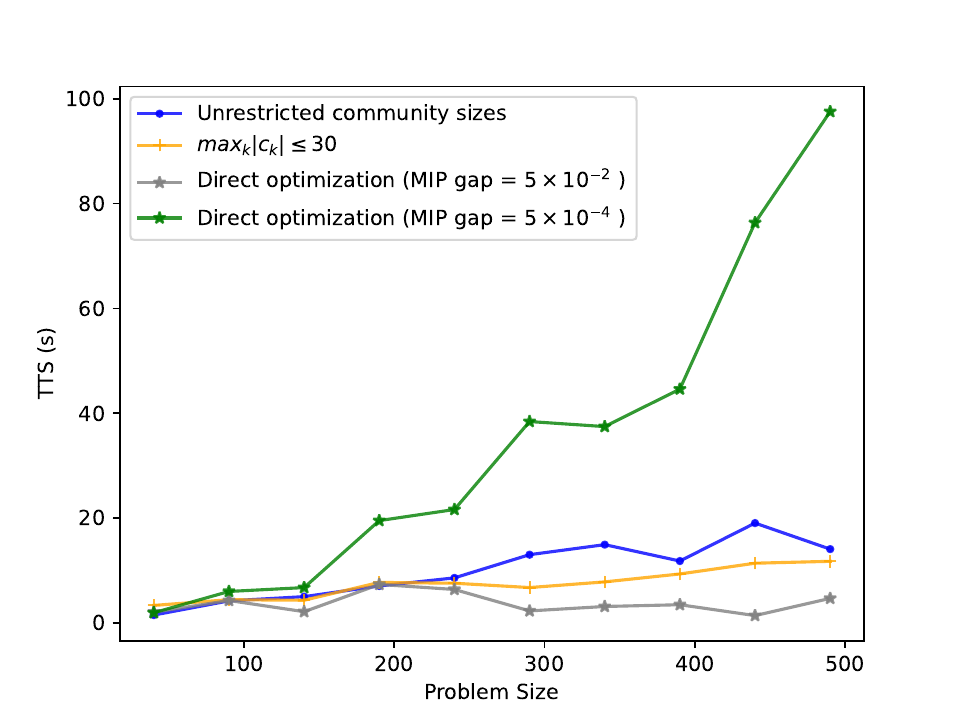}
    \put(5,68){(b)}
    \end{overpic}       
       \caption{Performance as a function of the problem size of the optimization of the risk minimization problem with quadratic constraint set at $a=0.9 \times \mathbf{x_{b}}^T\Sigma\mathbf{x_b}$ (Eq.~\eqref{eqn:quad_constraint_ps}). Panel (a) (top) shows the solution quality $H$ and panel (b) the time to solution (TTS), respectively. The lines marked with stars corresponds to directly optimizing the portfolio optimization, with the green-colored lines corresponding to the upper bound on MIP gap set at $5 \times 10^{-4}$ and the grey lines corresponding to a gap of $5 \times 10^{-2}$. The blue line (dots) and the orange (plus signs) refer to the proposed decomposition pipeline, where the former corresponds to the unrestricted community size and latter to a restricted size of the largest community size thresholded at $30$.}\label{fig:quad_combined}
\end{figure}

In our experiments, we obtained results using both Gurobi and  CPLEX. Our approach consistently obtained results with a better TTS using Gurobi, and as such we report benchmark results using Gurobi; see Appendix \ref{appendix:numerics} for a discussion. Note that the TTS reported corresponds to running the optimization of each subproblem sequentially, and each with the default MIP gap ($10^{-4}$). Additionally, the proposed decomposition pipeline allows parallelization of the optimization of each subproblem, improving the end-to-end runtime. The classical hardware on which our numerical experiments were run, together with more details, are specified in Appendix \ref{appendix:numerics}.

The results for the problem considered with linear constraints are shown in Fig.~\ref{fig:lin_combined}. We find that performance within $5\%$ of the optimum is possible while effectively reducing the system size by approximately $80\%$ (see Fig.~\ref{fig:size_largest}) and saving a factor of $3 \times$ in runtime for real-world problems with $\sim 1500$ assets. For a smaller problem size of $800$ assets, the end-to-end pipeline takes approximately $7$ seconds when we employ the community detection algorithm without any restrictions on the community sizes, and approximately $11$ seconds when we add an additional restriction on size of the communities to be less than $30$ -- about $2 \times $ faster than direct optimization without any decomposition, which needs more than $20$ seconds. In Figure~\ref{fig:lin_combined}, we observe the end-end pipeline with restriction on community size (orange plus signs) takes longer than unrestricted (blue dots) because the increment in time for finding smaller communities is more than the increment in the combined time taken for the optimization of each sub-problem. 

In order to better understand the optimization time of our decomposition pipeline, we plot the size of the largest community found in Fig.~\ref{fig:size_largest}. We observe an effective reduction of approximately $80\%$ in problem size, indicating that indeed the structure in the data enables us to decompose the problem, further explaining the advantage of using our pipeline over optimizing the full problem instances in time to solution. 

For the optimization problem with quadratic constraint shown in Eq.~\eqref{eqn:quad_constraint_ps} we perform a similar analysis in Fig.~\ref{fig:quad_combined}. Differently to the previously discussed experiments, we set the smaller MIP gap to be $5\times 10^{-4}$, because otherwise the TTS with default MIP gap $10^{-4}$ reaches a timeout of $1000$ for problem sizes as small as $n \approx 160$ (see Fig.~\ref{fig:TTS_quad_linear}). Although the decomposition pipeline significantly outperforms direct optimization with the MIP gap parameter set to $5\times 10^{-4}$, it obtains similar results in terms of the TTS with the direct optimization when the MIP gap tolerance is set to $5\times 10^{-2}$. We also plot the optimization score for the same in Fig.~\ref{fig:quad_combined}. We find that while effectively reducing the system size by approximately $80\%$, we are still within approximately $4\%$ of the optimum while taking approximately one tenth of the time taken by the solver with a MIP gap set at $5 \times 10^{-4}$ for real-world problems with approximately $500$ assets. However, for this problem, the Gurobi solver with an allowed larger MIP gap of $5 \times 10^{-2}$, does have a comparable TTS to the decomposition pipeline. It is important to note that  without the additional steps discussed in Sec.~\ref{sub:partition}, the objective values obtained through the decomposition pipeline can be lower than the objective values obtained via direct optimization, however the corresponding aggregated solutions become infeasible. We emphasize that in our results, we always use the rescaling described in Sec.~\ref{sub:partition} to enforce feasibility. 

\section{Discussion}

In summary, we have proposed and implemented a modular decomposition pipeline that leverages the community structure inherent to the underlying data to reduce large-scale PO problems into a family of much smaller (and independent) subproblems. Using real-world data from the Russell 3000 index, we have applied our decomposition logic to two different (NP-hard) PO problems, showing that an approximately $80\%$ reduction in problem size can be achieved while maintaining solution quality within a minimal error. Such a reduction in effective problem size opens the door for hybrid integration with near-term quantum devices, making large-scale PO problems with thousands of assets potentially compatible with quantum hardware with hundreds of qubits. It would be interesting to investigate a potential generalization of our framework towards a larger class of constrained PO problems and beyond in subsequent work. 

\section*{Code Availability} An open source demo version of our code is publicly available at 
\url{https://github.com/jpmorganchase/dcmppln}.

\section*{Acknowledgments}

We thank Brandon Augustino from the Global Technology Applied Research team at JPMorganChase for detailed reviews of the manuscript and Alex Buts and Raj Ganesan, also from Global Technology Applied Research team, for their contributions to the execution of the numerical simulations. We thank Eun Leonard, and Grant Chang from JPMorganChase for discussion on the quadratically constrained optimization problems. We thank Victor Bocking and Peter Sceusa for their support. 

\section*{Disclaimer}

This paper was prepared for informational purposes with contributions from the Global Technology Applied Research center of JPMorgan Chase \& Co. This paper is not a product of the Research Department of JPMorgan Chase \& Co. or its affiliates. Neither JPMorgan Chase \& Co. nor any of its affiliates makes any explicit or implied representation or warranty and none of them accept any liability in connection with this paper, including, without limitation, with respect to the completeness, accuracy, or reliability of the information contained herein and the potential legal, compliance, tax, or accounting effects thereof. This document is not intended as investment research or investment advice, or as a recommendation, offer, or solicitation for the purchase or sale of any security, financial instrument, financial product or service, or to be used in any way for evaluating the merits of participating in any transaction.

\bibliographystyle{apsrevtitle}
\bibliography{papers}

\begin{thebibliography}{67}
\expandafter\ifx\csname natexlab\endcsname\relax\def\natexlab#1{#1}\fi
\expandafter\ifx\csname bibnamefont\endcsname\relax
  \def\bibnamefont#1{#1}\fi
\expandafter\ifx\csname bibfnamefont\endcsname\relax
  \def\bibfnamefont#1{#1}\fi
\expandafter\ifx\csname citenamefont\endcsname\relax
  \def\citenamefont#1{#1}\fi
\expandafter\ifx\csname url\endcsname\relax
  \def\url#1{\texttt{#1}}\fi
\expandafter\ifx\csname urlprefix\endcsname\relax\def\urlprefix{URL }\fi
\providecommand{\bibinfo}[2]{#2}
\providecommand{\eprint}[2][]{\url{#2}}

\bibitem[{\citenamefont{Vaessens et~al.}(1998)\citenamefont{Vaessens, Aarts,
  and Lenstra}}]{vaessens1998local}
\bibinfo{author}{\bibfnamefont{R.~J.} \bibnamefont{Vaessens}},
  \bibinfo{author}{\bibfnamefont{E.~H.} \bibnamefont{Aarts}}, \bibnamefont{and}
  \bibinfo{author}{\bibfnamefont{J.~K.} \bibnamefont{Lenstra}},
  \emph{\bibinfo{title}{A local search template}}, \bibinfo{journal}{Computers
  \& Operations Research} \textbf{\bibinfo{volume}{25}}, \bibinfo{pages}{969}
  (\bibinfo{year}{1998}).

\bibitem[{\citenamefont{Wang et~al.}(2015)\citenamefont{Wang, Machta, and
  Katzgraber}}]{PhysRevE.92.013303}
\bibinfo{author}{\bibfnamefont{W.}~\bibnamefont{Wang}},
  \bibinfo{author}{\bibfnamefont{J.}~\bibnamefont{Machta}}, \bibnamefont{and}
  \bibinfo{author}{\bibfnamefont{H.~G.} \bibnamefont{Katzgraber}},
  \emph{\bibinfo{title}{Comparing monte carlo methods for finding ground states
  of ising spin glasses: Population annealing, simulated annealing, and
  parallel tempering}}, \bibinfo{journal}{Phys. Rev. E}
  \textbf{\bibinfo{volume}{92}}, \bibinfo{pages}{013303}
  (\bibinfo{year}{2015}),
  \urlprefix\url{https://link.aps.org/doi/10.1103/PhysRevE.92.013303}.

\bibitem[{\citenamefont{Hauke et~al.}(2020)\citenamefont{Hauke, Katzgraber,
  Lechner, Nishimori, and Oliver}}]{Hauke_2020}
\bibinfo{author}{\bibfnamefont{P.}~\bibnamefont{Hauke}},
  \bibinfo{author}{\bibfnamefont{H.~G.} \bibnamefont{Katzgraber}},
  \bibinfo{author}{\bibfnamefont{W.}~\bibnamefont{Lechner}},
  \bibinfo{author}{\bibfnamefont{H.}~\bibnamefont{Nishimori}},
  \bibnamefont{and} \bibinfo{author}{\bibfnamefont{W.~D.}
  \bibnamefont{Oliver}}, \emph{\bibinfo{title}{Perspectives of quantum
  annealing: methods and implementations}}, \bibinfo{journal}{Reports on
  Progress in Physics} \textbf{\bibinfo{volume}{83}}, \bibinfo{pages}{054401}
  (\bibinfo{year}{2020}), ISSN \bibinfo{issn}{1361-6633},
  \urlprefix\url{http://dx.doi.org/10.1088/1361-6633/ab85b8}.

\bibitem[{\citenamefont{Land and Doig}(1960)}]{LandDoig1960}
\bibinfo{author}{\bibfnamefont{A.~H.} \bibnamefont{Land}} \bibnamefont{and}
  \bibinfo{author}{\bibfnamefont{A.~G.} \bibnamefont{Doig}},
  \emph{\bibinfo{title}{An automatic method of solving discrete programming
  problems}}, \bibinfo{journal}{Econometrica} \textbf{\bibinfo{volume}{28}},
  \bibinfo{pages}{497} (\bibinfo{year}{1960}), \bibinfo{note}{accessed 3 June
  2024}, \urlprefix\url{https://doi.org/10.2307/1910129}.

\bibitem[{\citenamefont{Cornuejols and
  T{\"u}t{\"u}nc{\"u}}(2006)}]{Cornuejols2006Optimization}
\bibinfo{author}{\bibfnamefont{G.}~\bibnamefont{Cornuejols}} \bibnamefont{and}
  \bibinfo{author}{\bibfnamefont{R.}~\bibnamefont{T{\"u}t{\"u}nc{\"u}}},
  \emph{\bibinfo{title}{Optimization Methods in Finance}}
  (\bibinfo{publisher}{Carnegie Mellon University},
  \bibinfo{address}{Pittsburgh, PA}, \bibinfo{year}{2006}).

\bibitem[{\citenamefont{Cplex}(2009)}]{cplex2009v12}
\bibinfo{author}{\bibfnamefont{I.~I.} \bibnamefont{Cplex}},
  \emph{\bibinfo{title}{V12. 1: User’s manual for cplex}},
  \bibinfo{journal}{International Business Machines Corporation}
  \textbf{\bibinfo{volume}{46}}, \bibinfo{pages}{157} (\bibinfo{year}{2009}).

\bibitem[{\citenamefont{{Gurobi Optimization, LLC}}(2023)}]{gurobi}
\bibinfo{author}{\bibnamefont{{Gurobi Optimization, LLC}}},
  \emph{\bibinfo{title}{{Gurobi Optimizer Reference Manual}}}
  (\bibinfo{year}{2023}), \urlprefix\url{https://www.gurobi.com}.

\bibitem[{\citenamefont{Clausen}(1999)}]{clausen1999branch}
\bibinfo{author}{\bibfnamefont{J.}~\bibnamefont{Clausen}},
  \emph{\bibinfo{title}{Branch and bound algorithms-principles and examples}},
  \bibinfo{journal}{Department of computer science, University of Copenhagen}
  pp. \bibinfo{pages}{1--30} (\bibinfo{year}{1999}).

\bibitem[{\citenamefont{Brandao and Svore}(2017)}]{quant_speedup}
\bibinfo{author}{\bibfnamefont{F.~G.} \bibnamefont{Brandao}} \bibnamefont{and}
  \bibinfo{author}{\bibfnamefont{K.~M.} \bibnamefont{Svore}}, in
  \emph{\bibinfo{booktitle}{2017 IEEE 58th Annual Symposium on Foundations of
  Computer Science (FOCS)}} (\bibinfo{year}{2017}), pp.
  \bibinfo{pages}{415--426}.

\bibitem[{\citenamefont{Abbas et~al.}(2023)\citenamefont{Abbas, Ambainis,
  Augustino, B{\"a}rtschi, Buhrman, Coffrin, Cortiana, Dunjko, Egger, Elmegreen
  et~al.}}]{abbas2023quantum}
\bibinfo{author}{\bibfnamefont{A.}~\bibnamefont{Abbas}},
  \bibinfo{author}{\bibfnamefont{A.}~\bibnamefont{Ambainis}},
  \bibinfo{author}{\bibfnamefont{B.}~\bibnamefont{Augustino}},
  \bibinfo{author}{\bibfnamefont{A.}~\bibnamefont{B{\"a}rtschi}},
  \bibinfo{author}{\bibfnamefont{H.}~\bibnamefont{Buhrman}},
  \bibinfo{author}{\bibfnamefont{C.}~\bibnamefont{Coffrin}},
  \bibinfo{author}{\bibfnamefont{G.}~\bibnamefont{Cortiana}},
  \bibinfo{author}{\bibfnamefont{V.}~\bibnamefont{Dunjko}},
  \bibinfo{author}{\bibfnamefont{D.~J.} \bibnamefont{Egger}},
  \bibinfo{author}{\bibfnamefont{B.~G.} \bibnamefont{Elmegreen}},
  \bibnamefont{et~al.}, \emph{\bibinfo{title}{Quantum optimization: Potential,
  challenges, and the path forward}}, \bibinfo{journal}{arXiv preprint
  arXiv:2312.02279}  (\bibinfo{year}{2023}).

\bibitem[{\citenamefont{Dalzell
  et~al.}(2023{\natexlab{a}})\citenamefont{Dalzell, McArdle, Berta, Bienias,
  Chen, Gily{\'e}n, Hann, Kastoryano, Khabiboulline, Kubica
  et~al.}}]{dalzell2023quantum}
\bibinfo{author}{\bibfnamefont{A.~M.} \bibnamefont{Dalzell}},
  \bibinfo{author}{\bibfnamefont{S.}~\bibnamefont{McArdle}},
  \bibinfo{author}{\bibfnamefont{M.}~\bibnamefont{Berta}},
  \bibinfo{author}{\bibfnamefont{P.}~\bibnamefont{Bienias}},
  \bibinfo{author}{\bibfnamefont{C.-F.} \bibnamefont{Chen}},
  \bibinfo{author}{\bibfnamefont{A.}~\bibnamefont{Gily{\'e}n}},
  \bibinfo{author}{\bibfnamefont{C.~T.} \bibnamefont{Hann}},
  \bibinfo{author}{\bibfnamefont{M.~J.} \bibnamefont{Kastoryano}},
  \bibinfo{author}{\bibfnamefont{E.~T.} \bibnamefont{Khabiboulline}},
  \bibinfo{author}{\bibfnamefont{A.}~\bibnamefont{Kubica}},
  \bibnamefont{et~al.}, \emph{\bibinfo{title}{Quantum algorithms: A survey of
  applications and end-to-end complexities}}, \bibinfo{journal}{arXiv preprint
  arXiv:2310.03011}  (\bibinfo{year}{2023}{\natexlab{a}}).

\bibitem[{\citenamefont{Herman et~al.}(2022)\citenamefont{Herman, Googin, Liu,
  Galda, Safro, Sun, Pistoia, and Alexeev}}]{herman2022survey}
\bibinfo{author}{\bibfnamefont{D.}~\bibnamefont{Herman}},
  \bibinfo{author}{\bibfnamefont{C.}~\bibnamefont{Googin}},
  \bibinfo{author}{\bibfnamefont{X.}~\bibnamefont{Liu}},
  \bibinfo{author}{\bibfnamefont{A.}~\bibnamefont{Galda}},
  \bibinfo{author}{\bibfnamefont{I.}~\bibnamefont{Safro}},
  \bibinfo{author}{\bibfnamefont{Y.}~\bibnamefont{Sun}},
  \bibinfo{author}{\bibfnamefont{M.}~\bibnamefont{Pistoia}}, \bibnamefont{and}
  \bibinfo{author}{\bibfnamefont{Y.}~\bibnamefont{Alexeev}},
  \emph{\bibinfo{title}{A survey of quantum computing for finance}}
  (\bibinfo{year}{2022}), \eprint{2201.02773}.

\bibitem[{\citenamefont{He et~al.}(2024)\citenamefont{He, Chakrabarti, Herman,
  Kumar, Li, Minssen, Niroula, Ruslan, Sun, Sureshbabu et~al.}}]{DAC24_review}
\bibinfo{author}{\bibfnamefont{Z.}~\bibnamefont{He}},
  \bibinfo{author}{\bibfnamefont{S.}~\bibnamefont{Chakrabarti}},
  \bibinfo{author}{\bibfnamefont{D.}~\bibnamefont{Herman}},
  \bibinfo{author}{\bibfnamefont{N.}~\bibnamefont{Kumar}},
  \bibinfo{author}{\bibfnamefont{C.}~\bibnamefont{Li}},
  \bibinfo{author}{\bibfnamefont{P.}~\bibnamefont{Minssen}},
  \bibinfo{author}{\bibfnamefont{P.}~\bibnamefont{Niroula}},
  \bibinfo{author}{\bibfnamefont{S.}~\bibnamefont{Ruslan}},
  \bibinfo{author}{\bibfnamefont{Y.}~\bibnamefont{Sun}},
  \bibinfo{author}{\bibfnamefont{S.~H.} \bibnamefont{Sureshbabu}},
  \bibnamefont{et~al.}, in \emph{\bibinfo{booktitle}{accepted by 2024 61th
  ACM/IEEE Design Automation Conference (DAC)}} (\bibinfo{publisher}{IEEE},
  \bibinfo{year}{2024}), pp. \bibinfo{pages}{1--4}.

\bibitem[{\citenamefont{Hogg and Portnov}(2000)}]{Hogg2000}
\bibinfo{author}{\bibfnamefont{T.}~\bibnamefont{Hogg}} \bibnamefont{and}
  \bibinfo{author}{\bibfnamefont{D.}~\bibnamefont{Portnov}},
  \emph{\bibinfo{title}{Quantum optimization}}, \bibinfo{journal}{Information
  Sciences} \textbf{\bibinfo{volume}{128}}, \bibinfo{pages}{181–197}
  (\bibinfo{year}{2000}), ISSN \bibinfo{issn}{0020-0255}.

\bibitem[{\citenamefont{Farhi et~al.}(2014)\citenamefont{Farhi, Goldstone, and
  Gutmann}}]{farhi2014qaoa}
\bibinfo{author}{\bibfnamefont{E.}~\bibnamefont{Farhi}},
  \bibinfo{author}{\bibfnamefont{J.}~\bibnamefont{Goldstone}},
  \bibnamefont{and} \bibinfo{author}{\bibfnamefont{S.}~\bibnamefont{Gutmann}},
  \emph{\bibinfo{title}{A quantum approximate optimization algorithm}},
  \bibinfo{journal}{arXiv preprint arXiv:1411.4028}  (\bibinfo{year}{2014}).

\bibitem[{\citenamefont{Harrigan et~al.}(2021)\citenamefont{Harrigan, Sung,
  Neeley, Satzinger, Arute, Arya, Atalaya, Bardin, Barends, Boixo
  et~al.}}]{Harrigan2021}
\bibinfo{author}{\bibfnamefont{M.~P.} \bibnamefont{Harrigan}},
  \bibinfo{author}{\bibfnamefont{K.~J.} \bibnamefont{Sung}},
  \bibinfo{author}{\bibfnamefont{M.}~\bibnamefont{Neeley}},
  \bibinfo{author}{\bibfnamefont{K.~J.} \bibnamefont{Satzinger}},
  \bibinfo{author}{\bibfnamefont{F.}~\bibnamefont{Arute}},
  \bibinfo{author}{\bibfnamefont{K.}~\bibnamefont{Arya}},
  \bibinfo{author}{\bibfnamefont{J.}~\bibnamefont{Atalaya}},
  \bibinfo{author}{\bibfnamefont{J.~C.} \bibnamefont{Bardin}},
  \bibinfo{author}{\bibfnamefont{R.}~\bibnamefont{Barends}},
  \bibinfo{author}{\bibfnamefont{S.}~\bibnamefont{Boixo}},
  \bibnamefont{et~al.}, \emph{\bibinfo{title}{Quantum approximate optimization
  of non-planar graph problems on a planar superconducting processor}},
  \bibinfo{journal}{Nature Physics} \textbf{\bibinfo{volume}{17}},
  \bibinfo{pages}{332–336} (\bibinfo{year}{2021}), ISSN
  \bibinfo{issn}{1745-2481},
  \urlprefix\url{http://dx.doi.org/10.1038/s41567-020-01105-y}.

\bibitem[{\citenamefont{Pelofske et~al.}(2024)\citenamefont{Pelofske,
  B\"{a}rtschi, and Eidenbenz}}]{Pelofske2024}
\bibinfo{author}{\bibfnamefont{E.}~\bibnamefont{Pelofske}},
  \bibinfo{author}{\bibfnamefont{A.}~\bibnamefont{B\"{a}rtschi}},
  \bibnamefont{and}
  \bibinfo{author}{\bibfnamefont{S.}~\bibnamefont{Eidenbenz}},
  \emph{\bibinfo{title}{Short-depth qaoa circuits and quantum annealing on
  higher-order ising models}}, \bibinfo{journal}{npj Quantum Information}
  \textbf{\bibinfo{volume}{10}} (\bibinfo{year}{2024}), ISSN
  \bibinfo{issn}{2056-6387},
  \urlprefix\url{http://dx.doi.org/10.1038/s41534-024-00825-w}.

\bibitem[{\citenamefont{Shaydulin et~al.}(2024)\citenamefont{Shaydulin, Li,
  Chakrabarti, DeCross, Herman, Kumar, Larson, Lykov, Minssen, Sun
  et~al.}}]{shaydulin2024evidence}
\bibinfo{author}{\bibfnamefont{R.}~\bibnamefont{Shaydulin}},
  \bibinfo{author}{\bibfnamefont{C.}~\bibnamefont{Li}},
  \bibinfo{author}{\bibfnamefont{S.}~\bibnamefont{Chakrabarti}},
  \bibinfo{author}{\bibfnamefont{M.}~\bibnamefont{DeCross}},
  \bibinfo{author}{\bibfnamefont{D.}~\bibnamefont{Herman}},
  \bibinfo{author}{\bibfnamefont{N.}~\bibnamefont{Kumar}},
  \bibinfo{author}{\bibfnamefont{J.}~\bibnamefont{Larson}},
  \bibinfo{author}{\bibfnamefont{D.}~\bibnamefont{Lykov}},
  \bibinfo{author}{\bibfnamefont{P.}~\bibnamefont{Minssen}},
  \bibinfo{author}{\bibfnamefont{Y.}~\bibnamefont{Sun}}, \bibnamefont{et~al.},
  \emph{\bibinfo{title}{Evidence of scaling advantage for the quantum
  approximate optimization algorithm on a classically intractable problem}},
  \bibinfo{journal}{Science Advances} \textbf{\bibinfo{volume}{10}},
  \bibinfo{pages}{eadm6761} (\bibinfo{year}{2024}).

\bibitem[{\citenamefont{Niroula et~al.}(2022)\citenamefont{Niroula, Shaydulin,
  Yalovetzky, Minssen, Herman, Hu, and Pistoia}}]{niroula2022constrained}
\bibinfo{author}{\bibfnamefont{P.}~\bibnamefont{Niroula}},
  \bibinfo{author}{\bibfnamefont{R.}~\bibnamefont{Shaydulin}},
  \bibinfo{author}{\bibfnamefont{R.}~\bibnamefont{Yalovetzky}},
  \bibinfo{author}{\bibfnamefont{P.}~\bibnamefont{Minssen}},
  \bibinfo{author}{\bibfnamefont{D.}~\bibnamefont{Herman}},
  \bibinfo{author}{\bibfnamefont{S.}~\bibnamefont{Hu}}, \bibnamefont{and}
  \bibinfo{author}{\bibfnamefont{M.}~\bibnamefont{Pistoia}},
  \emph{\bibinfo{title}{Constrained quantum optimization for extractive
  summarization on a trapped-ion quantum computer}},
  \bibinfo{journal}{Scientific Reports} \textbf{\bibinfo{volume}{12}},
  \bibinfo{pages}{17171} (\bibinfo{year}{2022}).

\bibitem[{\citenamefont{Buonaiuto et~al.}(2023)\citenamefont{Buonaiuto,
  Gargiulo, De~Pietro, Esposito, and Pota}}]{Buonaiuto2023}
\bibinfo{author}{\bibfnamefont{G.}~\bibnamefont{Buonaiuto}},
  \bibinfo{author}{\bibfnamefont{F.}~\bibnamefont{Gargiulo}},
  \bibinfo{author}{\bibfnamefont{G.}~\bibnamefont{De~Pietro}},
  \bibinfo{author}{\bibfnamefont{M.}~\bibnamefont{Esposito}}, \bibnamefont{and}
  \bibinfo{author}{\bibfnamefont{M.}~\bibnamefont{Pota}},
  \emph{\bibinfo{title}{Best practices for portfolio optimization by quantum
  computing, experimented on real quantum devices}},
  \bibinfo{journal}{Scientific Reports} \textbf{\bibinfo{volume}{13}},
  \bibinfo{pages}{19434} (\bibinfo{year}{2023}),
  \urlprefix\url{https://doi.org/10.1038/s41598-023-45392-w}.

\bibitem[{\citenamefont{He et~al.}(2023)\citenamefont{He, Shaydulin,
  Chakrabarti, Herman, Li, Sun, and Pistoia}}]{he2023alignment}
\bibinfo{author}{\bibfnamefont{Z.}~\bibnamefont{He}},
  \bibinfo{author}{\bibfnamefont{R.}~\bibnamefont{Shaydulin}},
  \bibinfo{author}{\bibfnamefont{S.}~\bibnamefont{Chakrabarti}},
  \bibinfo{author}{\bibfnamefont{D.}~\bibnamefont{Herman}},
  \bibinfo{author}{\bibfnamefont{C.}~\bibnamefont{Li}},
  \bibinfo{author}{\bibfnamefont{Y.}~\bibnamefont{Sun}}, \bibnamefont{and}
  \bibinfo{author}{\bibfnamefont{M.}~\bibnamefont{Pistoia}},
  \emph{\bibinfo{title}{Alignment between initial state and mixer improves qaoa
  performance for constrained optimization}}, \bibinfo{journal}{npj Quantum
  Information} \textbf{\bibinfo{volume}{9}}, \bibinfo{pages}{121}
  (\bibinfo{year}{2023}).

\bibitem[{\citenamefont{Herman et~al.}(2023)\citenamefont{Herman, Shaydulin,
  Sun, Chakrabarti, Hu, Minssen, Rattew, Yalovetzky, and
  Pistoia}}]{herman2023constrained}
\bibinfo{author}{\bibfnamefont{D.}~\bibnamefont{Herman}},
  \bibinfo{author}{\bibfnamefont{R.}~\bibnamefont{Shaydulin}},
  \bibinfo{author}{\bibfnamefont{Y.}~\bibnamefont{Sun}},
  \bibinfo{author}{\bibfnamefont{S.}~\bibnamefont{Chakrabarti}},
  \bibinfo{author}{\bibfnamefont{S.}~\bibnamefont{Hu}},
  \bibinfo{author}{\bibfnamefont{P.}~\bibnamefont{Minssen}},
  \bibinfo{author}{\bibfnamefont{A.}~\bibnamefont{Rattew}},
  \bibinfo{author}{\bibfnamefont{R.}~\bibnamefont{Yalovetzky}},
  \bibnamefont{and} \bibinfo{author}{\bibfnamefont{M.}~\bibnamefont{Pistoia}},
  \emph{\bibinfo{title}{Constrained optimization via quantum zeno dynamics}},
  \bibinfo{journal}{Communications Physics} \textbf{\bibinfo{volume}{6}},
  \bibinfo{pages}{219} (\bibinfo{year}{2023}).

\bibitem[{\citenamefont{Sureshbabu et~al.}(2024)\citenamefont{Sureshbabu,
  Herman, Shaydulin, Basso, Chakrabarti, Sun, and
  Pistoia}}]{sureshbabu2024parameter}
\bibinfo{author}{\bibfnamefont{S.~H.} \bibnamefont{Sureshbabu}},
  \bibinfo{author}{\bibfnamefont{D.}~\bibnamefont{Herman}},
  \bibinfo{author}{\bibfnamefont{R.}~\bibnamefont{Shaydulin}},
  \bibinfo{author}{\bibfnamefont{J.}~\bibnamefont{Basso}},
  \bibinfo{author}{\bibfnamefont{S.}~\bibnamefont{Chakrabarti}},
  \bibinfo{author}{\bibfnamefont{Y.}~\bibnamefont{Sun}}, \bibnamefont{and}
  \bibinfo{author}{\bibfnamefont{M.}~\bibnamefont{Pistoia}},
  \emph{\bibinfo{title}{Parameter setting in quantum approximate optimization
  of weighted problems}}, \bibinfo{journal}{Quantum}
  \textbf{\bibinfo{volume}{8}}, \bibinfo{pages}{1231} (\bibinfo{year}{2024}).

\bibitem[{\citenamefont{Lang et~al.}(2022)\citenamefont{Lang, Zielinski, and
  Feld}}]{app122312288}
\bibinfo{author}{\bibfnamefont{J.}~\bibnamefont{Lang}},
  \bibinfo{author}{\bibfnamefont{S.}~\bibnamefont{Zielinski}},
  \bibnamefont{and} \bibinfo{author}{\bibfnamefont{S.}~\bibnamefont{Feld}},
  \emph{\bibinfo{title}{Strategic portfolio optimization using simulated,
  digital, and quantum annealing}}, \bibinfo{journal}{Applied Sciences}
  \textbf{\bibinfo{volume}{12}} (\bibinfo{year}{2022}), ISSN
  \bibinfo{issn}{2076-3417},
  \urlprefix\url{https://www.mdpi.com/2076-3417/12/23/12288}.

\bibitem[{\citenamefont{Cohen et~al.}(2008)\citenamefont{Cohen, Khan, and
  Alexander}}]{cohen2008portfolio}
\bibinfo{author}{\bibfnamefont{J.}~\bibnamefont{Cohen}},
  \bibinfo{author}{\bibfnamefont{A.}~\bibnamefont{Khan}}, \bibnamefont{and}
  \bibinfo{author}{\bibfnamefont{C.}~\bibnamefont{Alexander}},
  \emph{\bibinfo{title}{Portfolio optimization of 60 stocks using classical and
  quantum algorithms (2020)}}, \bibinfo{journal}{arXiv preprint
  arXiv:2008.08669}  (\bibinfo{year}{2008}), \eprint{arXiv:2008.08669}.

\bibitem[{\citenamefont{Harrow et~al.}(2009)\citenamefont{Harrow, Hassidim, and
  Lloyd}}]{harrow2009quantum}
\bibinfo{author}{\bibfnamefont{A.~W.} \bibnamefont{Harrow}},
  \bibinfo{author}{\bibfnamefont{A.}~\bibnamefont{Hassidim}}, \bibnamefont{and}
  \bibinfo{author}{\bibfnamefont{S.}~\bibnamefont{Lloyd}},
  \emph{\bibinfo{title}{Quantum algorithm for linear systems of equations}},
  \bibinfo{journal}{Physical Review Letters} \textbf{\bibinfo{volume}{103}},
  \bibinfo{pages}{150502} (\bibinfo{year}{2009}).

\bibitem[{\citenamefont{Childs and Wiebe}(2012)}]{childs2012hamiltonian}
\bibinfo{author}{\bibfnamefont{A.~M.} \bibnamefont{Childs}} \bibnamefont{and}
  \bibinfo{author}{\bibfnamefont{N.}~\bibnamefont{Wiebe}},
  \emph{\bibinfo{title}{{Hamiltonian Simulation using linear combinations of
  unitary operations}}}, \bibinfo{journal}{Quantum Information and Computation}
  \textbf{\bibinfo{volume}{12}}, \bibinfo{pages}{901–924}
  (\bibinfo{year}{2012}), ISSN \bibinfo{issn}{1533-7146}.

\bibitem[{\citenamefont{Chakraborty et~al.}(2019)\citenamefont{Chakraborty,
  Gily{\'e}n, and Jeffery}}]{chakraborty2018power}
\bibinfo{author}{\bibfnamefont{S.}~\bibnamefont{Chakraborty}},
  \bibinfo{author}{\bibfnamefont{A.}~\bibnamefont{Gily{\'e}n}},
  \bibnamefont{and} \bibinfo{author}{\bibfnamefont{S.}~\bibnamefont{Jeffery}},
  in \emph{\bibinfo{booktitle}{46th International Colloquium on Automata,
  Languages, and Programming (ICALP 2019)}}, edited by
  \bibinfo{editor}{\bibfnamefont{C.}~\bibnamefont{Baier}},
  \bibinfo{editor}{\bibfnamefont{I.}~\bibnamefont{Chatzigiannakis}},
  \bibinfo{editor}{\bibfnamefont{P.}~\bibnamefont{Flocchini}},
  \bibnamefont{and} \bibinfo{editor}{\bibfnamefont{S.}~\bibnamefont{Leonardi}}
  (\bibinfo{publisher}{Schloss Dagstuhl--Leibniz-Zentrum fuer Informatik},
  \bibinfo{address}{Dagstuhl, Germany}, \bibinfo{year}{2019}), vol.
  \bibinfo{volume}{132}, pp. \bibinfo{pages}{33:1--33:14}.

\bibitem[{\citenamefont{Rebentrost and Lloyd}(2018)}]{rebentrost2018quantum}
\bibinfo{author}{\bibfnamefont{P.}~\bibnamefont{Rebentrost}} \bibnamefont{and}
  \bibinfo{author}{\bibfnamefont{S.}~\bibnamefont{Lloyd}},
  \emph{\bibinfo{title}{Quantum computational finance: quantum algorithm for
  portfolio optimization}}, \bibinfo{journal}{arXiv preprint arXiv:1811.03975}
  (\bibinfo{year}{2018}).

\bibitem[{\citenamefont{Yalovetzky et~al.}(2024)\citenamefont{Yalovetzky,
  Minssen, Herman, and Pistoia}}]{yalovetzky2024hybrid}
\bibinfo{author}{\bibfnamefont{R.}~\bibnamefont{Yalovetzky}},
  \bibinfo{author}{\bibfnamefont{P.}~\bibnamefont{Minssen}},
  \bibinfo{author}{\bibfnamefont{D.}~\bibnamefont{Herman}}, \bibnamefont{and}
  \bibinfo{author}{\bibfnamefont{M.}~\bibnamefont{Pistoia}},
  \emph{\bibinfo{title}{Solving linear systems on quantum hardware with hybrid
  hhl++}}, \bibinfo{journal}{Nature Scientific Reports}
  \textbf{\bibinfo{volume}{14}} (\bibinfo{year}{2024}).

\bibitem[{\citenamefont{Dalzell
  et~al.}(2023{\natexlab{b}})\citenamefont{Dalzell, Clader, Salton, Berta, Lin,
  Bader, Stamatopoulos, Schuetz, Brand\~ao, Katzgraber
  et~al.}}]{PRXQuantum.4.040325}
\bibinfo{author}{\bibfnamefont{A.~M.} \bibnamefont{Dalzell}},
  \bibinfo{author}{\bibfnamefont{B.~D.} \bibnamefont{Clader}},
  \bibinfo{author}{\bibfnamefont{G.}~\bibnamefont{Salton}},
  \bibinfo{author}{\bibfnamefont{M.}~\bibnamefont{Berta}},
  \bibinfo{author}{\bibfnamefont{C.~Y.-Y.} \bibnamefont{Lin}},
  \bibinfo{author}{\bibfnamefont{D.~A.} \bibnamefont{Bader}},
  \bibinfo{author}{\bibfnamefont{N.}~\bibnamefont{Stamatopoulos}},
  \bibinfo{author}{\bibfnamefont{M.~J.~A.} \bibnamefont{Schuetz}},
  \bibinfo{author}{\bibfnamefont{F.~G. S.~L.} \bibnamefont{Brand\~ao}},
  \bibinfo{author}{\bibfnamefont{H.~G.} \bibnamefont{Katzgraber}},
  \bibnamefont{et~al.}, \emph{\bibinfo{title}{End-to-end resource analysis for
  quantum interior-point methods and portfolio optimization}},
  \bibinfo{journal}{PRX Quantum} \textbf{\bibinfo{volume}{4}},
  \bibinfo{pages}{040325} (\bibinfo{year}{2023}{\natexlab{b}}),
  \urlprefix\url{https://link.aps.org/doi/10.1103/PRXQuantum.4.040325}.

\bibitem[{\citenamefont{Kerenidis et~al.}(2019)\citenamefont{Kerenidis,
  Prakash, and Szil{\'a}gyi}}]{kerenidis2019quantum}
\bibinfo{author}{\bibfnamefont{I.}~\bibnamefont{Kerenidis}},
  \bibinfo{author}{\bibfnamefont{A.}~\bibnamefont{Prakash}}, \bibnamefont{and}
  \bibinfo{author}{\bibfnamefont{D.}~\bibnamefont{Szil{\'a}gyi}}, in
  \emph{\bibinfo{booktitle}{Proceedings of the 1st ACM Conference on Advances
  in Financial Technologies}} (\bibinfo{year}{2019}), pp.
  \bibinfo{pages}{147--155}.

\bibitem[{\citenamefont{Montanaro}(2020)}]{B_B_Ashley}
\bibinfo{author}{\bibfnamefont{A.}~\bibnamefont{Montanaro}},
  \emph{\bibinfo{title}{Quantum speedup of branch-and-bound algorithms}},
  \bibinfo{journal}{Phys. Rev. Res.} \textbf{\bibinfo{volume}{2}},
  \bibinfo{pages}{013056} (\bibinfo{year}{2020}),
  \urlprefix\url{https://link.aps.org/doi/10.1103/PhysRevResearch.2.013056}.

\bibitem[{\citenamefont{Chakrabarti et~al.}(2022)\citenamefont{Chakrabarti,
  Minssen, Yalovetzky, and Pistoia}}]{chakrabarti2022universal}
\bibinfo{author}{\bibfnamefont{S.}~\bibnamefont{Chakrabarti}},
  \bibinfo{author}{\bibfnamefont{P.}~\bibnamefont{Minssen}},
  \bibinfo{author}{\bibfnamefont{R.}~\bibnamefont{Yalovetzky}},
  \bibnamefont{and} \bibinfo{author}{\bibfnamefont{M.}~\bibnamefont{Pistoia}},
  \emph{\bibinfo{title}{Universal quantum speedup for branch-and-bound,
  branch-and-cut, and tree-search algorithms}} (\bibinfo{year}{2022}),
  \eprint{2210.03210}.

\bibitem[{\citenamefont{MacMahon and Garlaschelli}(2015)}]{MacMahon_2015}
\bibinfo{author}{\bibfnamefont{M.}~\bibnamefont{MacMahon}} \bibnamefont{and}
  \bibinfo{author}{\bibfnamefont{D.}~\bibnamefont{Garlaschelli}},
  \emph{\bibinfo{title}{Community detection for correlation matrices}},
  \bibinfo{journal}{Physical Review X} \textbf{\bibinfo{volume}{5}}
  (\bibinfo{year}{2015}),
  \urlprefix\url{https://doi.org/10.1103%2Fphysrevx.5.021006}.

\bibitem[{\citenamefont{David~Disatnik}(2012)}]{block_po}
\bibinfo{author}{\bibfnamefont{S.~K.} \bibnamefont{David~Disatnik}},
  \emph{\bibinfo{title}{Portfolio optimization using a block structure for the
  covariance matrix}}, \bibinfo{journal}{Journal of Business Finance \&
  Accounting} \textbf{\bibinfo{volume}{39}}, \bibinfo{pages}{806}
  (\bibinfo{year}{2012}), ISSN \bibinfo{issn}{0306-686X}.

\bibitem[{\citenamefont{Chan et~al.}(1999)\citenamefont{Chan, Karceski, and
  Lakonishok}}]{chan1999portfolio}
\bibinfo{author}{\bibfnamefont{L.~K.} \bibnamefont{Chan}},
  \bibinfo{author}{\bibfnamefont{J.}~\bibnamefont{Karceski}}, \bibnamefont{and}
  \bibinfo{author}{\bibfnamefont{J.}~\bibnamefont{Lakonishok}},
  \bibinfo{type}{NBER Working Paper} \bibinfo{number}{7039},
  \bibinfo{institution}{National Bureau of Economic Research}
  (\bibinfo{year}{1999}), \bibinfo{note}{jEL No. G11, G12},
  \urlprefix\url{http://www.nber.org/papers/w7039}.

\bibitem[{\citenamefont{León et~al.}(2017)\citenamefont{León, Aragón,
  Sandoval, Hernández, Arévalo, and Niño}}]{LEON20171334}
\bibinfo{author}{\bibfnamefont{D.}~\bibnamefont{León}},
  \bibinfo{author}{\bibfnamefont{A.}~\bibnamefont{Aragón}},
  \bibinfo{author}{\bibfnamefont{J.}~\bibnamefont{Sandoval}},
  \bibinfo{author}{\bibfnamefont{G.}~\bibnamefont{Hernández}},
  \bibinfo{author}{\bibfnamefont{A.}~\bibnamefont{Arévalo}}, \bibnamefont{and}
  \bibinfo{author}{\bibfnamefont{J.}~\bibnamefont{Niño}},
  \emph{\bibinfo{title}{Clustering algorithms for risk-adjusted portfolio
  construction}}, \bibinfo{journal}{Procedia Computer Science}
  \textbf{\bibinfo{volume}{108}}, \bibinfo{pages}{1334} (\bibinfo{year}{2017}),
  ISSN \bibinfo{issn}{1877-0509}, \bibinfo{note}{international Conference on
  Computational Science, ICCS 2017, 12-14 June 2017, Zurich, Switzerland},
  \urlprefix\url{https://www.sciencedirect.com/science/article/pii/S187705091730772X}.

\bibitem[{\citenamefont{Rahmaniani et~al.}(2017)\citenamefont{Rahmaniani,
  Crainic, Gendreau, and Rei}}]{rahmaniani2017benders}
\bibinfo{author}{\bibfnamefont{R.}~\bibnamefont{Rahmaniani}},
  \bibinfo{author}{\bibfnamefont{T.~G.} \bibnamefont{Crainic}},
  \bibinfo{author}{\bibfnamefont{M.}~\bibnamefont{Gendreau}}, \bibnamefont{and}
  \bibinfo{author}{\bibfnamefont{W.}~\bibnamefont{Rei}},
  \emph{\bibinfo{title}{The benders decomposition algorithm: A literature
  review}}, \bibinfo{journal}{European Journal of Operational Research}
  \textbf{\bibinfo{volume}{259}}, \bibinfo{pages}{801} (\bibinfo{year}{2017}).

\bibitem[{\citenamefont{Benders}(2005)}]{benders2005partitioning}
\bibinfo{author}{\bibfnamefont{J.}~\bibnamefont{Benders}},
  \emph{\bibinfo{title}{Partitioning procedures for solving mixed-variables
  programming problems.}}, \bibinfo{journal}{Computational Management Science}
  \textbf{\bibinfo{volume}{2}} (\bibinfo{year}{2005}).

\bibitem[{\citenamefont{Dees et~al.}(2020)\citenamefont{Dees, Stanković,
  Constantinides, and Mandic}}]{port_cuts}
\bibinfo{author}{\bibfnamefont{B.~S.} \bibnamefont{Dees}},
  \bibinfo{author}{\bibfnamefont{L.}~\bibnamefont{Stanković}},
  \bibinfo{author}{\bibfnamefont{A.~G.} \bibnamefont{Constantinides}},
  \bibnamefont{and} \bibinfo{author}{\bibfnamefont{D.~P.}
  \bibnamefont{Mandic}}, in \emph{\bibinfo{booktitle}{ICASSP 2020 - 2020 IEEE
  International Conference on Acoustics, Speech and Signal Processing
  (ICASSP)}} (\bibinfo{year}{2020}), pp. \bibinfo{pages}{8454--8458}.

\bibitem[{\citenamefont{Arroyo et~al.}(2021)\citenamefont{Arroyo, Scalzo,
  Stankovic, and Mandic}}]{arroyo2021dynamicportfoliocutsspectral}
\bibinfo{author}{\bibfnamefont{A.}~\bibnamefont{Arroyo}},
  \bibinfo{author}{\bibfnamefont{B.}~\bibnamefont{Scalzo}},
  \bibinfo{author}{\bibfnamefont{L.}~\bibnamefont{Stankovic}},
  \bibnamefont{and} \bibinfo{author}{\bibfnamefont{D.~P.}
  \bibnamefont{Mandic}}, \emph{\bibinfo{title}{Dynamic portfolio cuts: A
  spectral approach to graph-theoretic diversification}}
  (\bibinfo{year}{2021}), \eprint{2106.03417},
  \urlprefix\url{https://arxiv.org/abs/2106.03417}.

\bibitem[{\citenamefont{Shaydulin et~al.}(2018)\citenamefont{Shaydulin,
  Ushijima-Mwesigwa, Safro, Mniszewski, and Alexeev}}]{shaydulincommunity}
\bibinfo{author}{\bibfnamefont{R.}~\bibnamefont{Shaydulin}},
  \bibinfo{author}{\bibfnamefont{H.}~\bibnamefont{Ushijima-Mwesigwa}},
  \bibinfo{author}{\bibfnamefont{I.}~\bibnamefont{Safro}},
  \bibinfo{author}{\bibfnamefont{S.}~\bibnamefont{Mniszewski}},
  \bibnamefont{and} \bibinfo{author}{\bibfnamefont{Y.}~\bibnamefont{Alexeev}},
  \emph{\bibinfo{title}{Community detection across emerging quantum
  architectures}}, \bibinfo{journal}{Proceedings of the 3rd International
  Workshop on Post Moore’s Era Supercomputing (in conjunction with
  Supercomputing ’18)} pp. \bibinfo{pages}{12--14} (\bibinfo{year}{2018}).

\bibitem[{\citenamefont{Shaydulin
  et~al.}(2019{\natexlab{a}})\citenamefont{Shaydulin, Ushijima‐Mwesigwa,
  Safro, Mniszewski, and Alexeev}}]{Shaydulin2019}
\bibinfo{author}{\bibfnamefont{R.}~\bibnamefont{Shaydulin}},
  \bibinfo{author}{\bibfnamefont{H.}~\bibnamefont{Ushijima‐Mwesigwa}},
  \bibinfo{author}{\bibfnamefont{I.}~\bibnamefont{Safro}},
  \bibinfo{author}{\bibfnamefont{S.}~\bibnamefont{Mniszewski}},
  \bibnamefont{and} \bibinfo{author}{\bibfnamefont{Y.}~\bibnamefont{Alexeev}},
  \emph{\bibinfo{title}{Network community detection on small quantum
  computers}}, \bibinfo{journal}{Advanced Quantum Technologies}
  \textbf{\bibinfo{volume}{2}} (\bibinfo{year}{2019}{\natexlab{a}}), ISSN
  \bibinfo{issn}{2511-9044},
  \urlprefix\url{http://dx.doi.org/10.1002/qute.201900029}.

\bibitem[{\citenamefont{Shaydulin
  et~al.}(2019{\natexlab{b}})\citenamefont{Shaydulin, Ushijima-Mwesigwa, Negre,
  Safro, Mniszewski, and Alexeev}}]{Shaydulin20192}
\bibinfo{author}{\bibfnamefont{R.}~\bibnamefont{Shaydulin}},
  \bibinfo{author}{\bibfnamefont{H.}~\bibnamefont{Ushijima-Mwesigwa}},
  \bibinfo{author}{\bibfnamefont{C.~F.~A.} \bibnamefont{Negre}},
  \bibinfo{author}{\bibfnamefont{I.}~\bibnamefont{Safro}},
  \bibinfo{author}{\bibfnamefont{S.~M.} \bibnamefont{Mniszewski}},
  \bibnamefont{and} \bibinfo{author}{\bibfnamefont{Y.}~\bibnamefont{Alexeev}},
  \emph{\bibinfo{title}{A hybrid approach for solving optimization problems on
  small quantum computers}}, \bibinfo{journal}{Computer}
  \textbf{\bibinfo{volume}{52}}, \bibinfo{pages}{18–26}
  (\bibinfo{year}{2019}{\natexlab{b}}), ISSN \bibinfo{issn}{1558-0814},
  \urlprefix\url{http://dx.doi.org/10.1109/MC.2019.2908942}.

\bibitem[{\citenamefont{Ushijima-Mwesigwa
  et~al.}(2021)\citenamefont{Ushijima-Mwesigwa, Shaydulin, Negre, Mniszewski,
  Alexeev, and Safro}}]{UshijimaMwesigwa2021}
\bibinfo{author}{\bibfnamefont{H.}~\bibnamefont{Ushijima-Mwesigwa}},
  \bibinfo{author}{\bibfnamefont{R.}~\bibnamefont{Shaydulin}},
  \bibinfo{author}{\bibfnamefont{C.~F.~A.} \bibnamefont{Negre}},
  \bibinfo{author}{\bibfnamefont{S.~M.} \bibnamefont{Mniszewski}},
  \bibinfo{author}{\bibfnamefont{Y.}~\bibnamefont{Alexeev}}, \bibnamefont{and}
  \bibinfo{author}{\bibfnamefont{I.}~\bibnamefont{Safro}},
  \emph{\bibinfo{title}{Multilevel combinatorial optimization across quantum
  architectures}}, \bibinfo{journal}{ACM Transactions on Quantum Computing}
  \textbf{\bibinfo{volume}{2}}, \bibinfo{pages}{1–29} (\bibinfo{year}{2021}),
  ISSN \bibinfo{issn}{2643-6817},
  \urlprefix\url{http://dx.doi.org/10.1145/3425607}.

\bibitem[{\citenamefont{Wishart}(1928)}]{Wishart1928}
\bibinfo{author}{\bibfnamefont{J.}~\bibnamefont{Wishart}},
  \emph{\bibinfo{title}{The generalised product moment distribution in samples
  from a normal multivariate population}}, \bibinfo{journal}{Biometrika}
  \textbf{\bibinfo{volume}{20A}}, \bibinfo{pages}{32} (\bibinfo{year}{1928}).

\bibitem[{\citenamefont{Bun et~al.}(2017{\natexlab{a}})\citenamefont{Bun,
  Bouchaud, and Potters}}]{Bun_2017}
\bibinfo{author}{\bibfnamefont{J.}~\bibnamefont{Bun}},
  \bibinfo{author}{\bibfnamefont{J.-P.} \bibnamefont{Bouchaud}},
  \bibnamefont{and} \bibinfo{author}{\bibfnamefont{M.}~\bibnamefont{Potters}},
  \emph{\bibinfo{title}{Cleaning large correlation matrices: Tools from random
  matrix theory}}, \bibinfo{journal}{Physics Reports}
  \textbf{\bibinfo{volume}{666}}, \bibinfo{pages}{1–109}
  (\bibinfo{year}{2017}{\natexlab{a}}), ISSN \bibinfo{issn}{0370-1573},
  \urlprefix\url{http://dx.doi.org/10.1016/j.physrep.2016.10.005}.

\bibitem[{\citenamefont{James and Stein}(1961)}]{james1961estimation}
\bibinfo{author}{\bibfnamefont{W.}~\bibnamefont{James}} \bibnamefont{and}
  \bibinfo{author}{\bibfnamefont{C.}~\bibnamefont{Stein}}, in
  \emph{\bibinfo{booktitle}{Proceedings of the Berkeley Symposium on
  Mathematical Statistics and Probability}} (\bibinfo{organization}{University
  of California Press}, \bibinfo{year}{1961}), vol. \bibinfo{volume}{4.1}, pp.
  \bibinfo{pages}{361--379}.

\bibitem[{\citenamefont{Stein}(1986)}]{Stein1986}
\bibinfo{author}{\bibfnamefont{C.}~\bibnamefont{Stein}},
  \emph{\bibinfo{title}{Lectures on the theory of estimation of many
  parameters}}, \bibinfo{journal}{Journal of Soviet Mathematics}
  \textbf{\bibinfo{volume}{34}}, \bibinfo{pages}{1373} (\bibinfo{year}{1986}),
  ISSN \bibinfo{issn}{1573-8795},
  \urlprefix\url{https://doi.org/10.1007/BF01085007}.

\bibitem[{\citenamefont{Ledoit and Wolf}(2004)}]{shrinkage_LEDOIT}
\bibinfo{author}{\bibfnamefont{O.}~\bibnamefont{Ledoit}} \bibnamefont{and}
  \bibinfo{author}{\bibfnamefont{M.}~\bibnamefont{Wolf}},
  \emph{\bibinfo{title}{A well-conditioned estimator for large-dimensional
  covariance matrices}}, \bibinfo{journal}{Journal of Multivariate Analysis}
  \textbf{\bibinfo{volume}{88}}, \bibinfo{pages}{365} (\bibinfo{year}{2004}),
  ISSN \bibinfo{issn}{0047-259X},
  \urlprefix\url{https://www.sciencedirect.com/science/article/pii/S0047259X03000964}.

\bibitem[{\citenamefont{Ledoit and Wolf}(2012)}]{LW2012}
\bibinfo{author}{\bibfnamefont{O.}~\bibnamefont{Ledoit}} \bibnamefont{and}
  \bibinfo{author}{\bibfnamefont{M.}~\bibnamefont{Wolf}},
  \emph{\bibinfo{title}{Nonlinear shrinkage estimation of large-dimensional
  covariance matrices}}, \bibinfo{journal}{Annals of Statistics}
  \textbf{\bibinfo{volume}{40}}, \bibinfo{pages}{1024} (\bibinfo{year}{2012}).

\bibitem[{\citenamefont{Ledoit and Wolf}(2015)}]{LW2015}
\bibinfo{author}{\bibfnamefont{O.}~\bibnamefont{Ledoit}} \bibnamefont{and}
  \bibinfo{author}{\bibfnamefont{M.}~\bibnamefont{Wolf}},
  \emph{\bibinfo{title}{Spectrum estimation: A unified framework for covariance
  matrix estimation and pca in large dimensions}}, \bibinfo{journal}{Journal of
  Multivariate Analysis} \textbf{\bibinfo{volume}{139}}, \bibinfo{pages}{360}
  (\bibinfo{year}{2015}).

\bibitem[{\citenamefont{Ledoit and Wolf}(2017)}]{LW2017b}
\bibinfo{author}{\bibfnamefont{O.}~\bibnamefont{Ledoit}} \bibnamefont{and}
  \bibinfo{author}{\bibfnamefont{M.}~\bibnamefont{Wolf}},
  \emph{\bibinfo{title}{Numerical implementation of the quest function}},
  \bibinfo{journal}{Computational Statistics \& Data Analysis}
  \textbf{\bibinfo{volume}{115}}, \bibinfo{pages}{199} (\bibinfo{year}{2017}).

\bibitem[{\citenamefont{Ledoit and Wolf}(2020)}]{LW2020}
\bibinfo{author}{\bibfnamefont{O.}~\bibnamefont{Ledoit}} \bibnamefont{and}
  \bibinfo{author}{\bibfnamefont{M.}~\bibnamefont{Wolf}},
  \emph{\bibinfo{title}{Analytical nonlinear shrinkage of large-dimensional
  covariance matrices}}, \bibinfo{journal}{Annals of Statistics}
  \textbf{\bibinfo{volume}{48}}, \bibinfo{pages}{3043} (\bibinfo{year}{2020}).

\bibitem[{\citenamefont{{Mar{\v{c}}enko} and
  {Pastur}}(1967)}]{1967SbMat...1..457M}
\bibinfo{author}{\bibfnamefont{V.~A.} \bibnamefont{{Mar{\v{c}}enko}}}
  \bibnamefont{and} \bibinfo{author}{\bibfnamefont{L.~A.}
  \bibnamefont{{Pastur}}}, \emph{\bibinfo{title}{{Distribution of Eigenvalues
  for Some Sets of Random Matrices}}}, \bibinfo{journal}{Sbornik: Mathematics}
  \textbf{\bibinfo{volume}{1}}, \bibinfo{pages}{457} (\bibinfo{year}{1967}).

\bibitem[{\citenamefont{Bouchaud and Potters}(2009)}]{bouchaud2009financial}
\bibinfo{author}{\bibfnamefont{J.~P.} \bibnamefont{Bouchaud}} \bibnamefont{and}
  \bibinfo{author}{\bibfnamefont{M.}~\bibnamefont{Potters}},
  \emph{\bibinfo{title}{Financial applications of random matrix theory: a short
  review}} (\bibinfo{year}{2009}), \eprint{0910.1205}.

\bibitem[{\citenamefont{Bun et~al.}(2017{\natexlab{b}})\citenamefont{Bun,
  Bouchaud, and Potters}}]{BUN20171}
\bibinfo{author}{\bibfnamefont{J.}~\bibnamefont{Bun}},
  \bibinfo{author}{\bibfnamefont{J.-P.} \bibnamefont{Bouchaud}},
  \bibnamefont{and} \bibinfo{author}{\bibfnamefont{M.}~\bibnamefont{Potters}},
  \emph{\bibinfo{title}{Cleaning large correlation matrices: Tools from random
  matrix theory}}, \bibinfo{journal}{Physics Reports}
  \textbf{\bibinfo{volume}{666}}, \bibinfo{pages}{1}
  (\bibinfo{year}{2017}{\natexlab{b}}), ISSN \bibinfo{issn}{0370-1573},
  \bibinfo{note}{cleaning large correlation matrices: tools from random matrix
  theory},
  \urlprefix\url{https://www.sciencedirect.com/science/article/pii/S0370157316303337}.

\bibitem[{\citenamefont{Bouchaud and Potters}(2003)}]{Bouchaud_Potters_2003}
\bibinfo{author}{\bibfnamefont{J.-P.} \bibnamefont{Bouchaud}} \bibnamefont{and}
  \bibinfo{author}{\bibfnamefont{M.}~\bibnamefont{Potters}},
  \emph{\bibinfo{title}{Theory of Financial Risk and Derivative Pricing: From
  Statistical Physics to Risk Management}} (\bibinfo{publisher}{Cambridge
  University Press}, \bibinfo{year}{2003}), \bibinfo{edition}{2nd} ed.

\bibitem[{\citenamefont{Laloux et~al.}(2000)\citenamefont{Laloux, Cizeau,
  Potters, and Bouchaud}}]{Laurent_RMT}
\bibinfo{author}{\bibfnamefont{L.}~\bibnamefont{Laloux}},
  \bibinfo{author}{\bibfnamefont{P.}~\bibnamefont{Cizeau}},
  \bibinfo{author}{\bibfnamefont{M.}~\bibnamefont{Potters}}, \bibnamefont{and}
  \bibinfo{author}{\bibfnamefont{J.-P.} \bibnamefont{Bouchaud}},
  \emph{\bibinfo{title}{Random matrix theory and financial correlations}},
  \bibinfo{journal}{International Journal of Theoretical and Applied Finance
  (IJTAF)} \textbf{\bibinfo{volume}{03}}, \bibinfo{pages}{391}
  (\bibinfo{year}{2000}),
  \urlprefix\url{https://EconPapers.repec.org/RePEc:wsi:ijtafx:v:03:y:2000:i:03:n:s0219024900000255}.

\bibitem[{\citenamefont{Donoho et~al.}(2018)\citenamefont{Donoho, Gavish, and
  Johnstone}}]{Donoho2018Optimal}
\bibinfo{author}{\bibfnamefont{D.~L.} \bibnamefont{Donoho}},
  \bibinfo{author}{\bibfnamefont{M.}~\bibnamefont{Gavish}}, \bibnamefont{and}
  \bibinfo{author}{\bibfnamefont{I.~M.} \bibnamefont{Johnstone}},
  \emph{\bibinfo{title}{Optimal shrinkage of eigenvalues in the spiked
  covariance model}}, \bibinfo{journal}{Annals of Statistics}
  \textbf{\bibinfo{volume}{46}}, \bibinfo{pages}{1742} (\bibinfo{year}{2018}),
  \urlprefix\url{https://doi.org/10.1214/17-AOS1601}.

\bibitem[{\citenamefont{Newman}(2006)}]{Newman_2006}
\bibinfo{author}{\bibfnamefont{M.~E.~J.} \bibnamefont{Newman}},
  \emph{\bibinfo{title}{Modularity and community structure in networks}},
  \bibinfo{journal}{Proceedings of the National Academy of Sciences}
  \textbf{\bibinfo{volume}{103}}, \bibinfo{pages}{8577} (\bibinfo{year}{2006}),
  \urlprefix\url{https://doi.org/10.1073%2Fpnas.0601602103}.

\bibitem[{\citenamefont{Tsung et~al.}(2016)\citenamefont{Tsung, Ho, Chou, Lin,
  and Lee}}]{spectral_mod}
\bibinfo{author}{\bibfnamefont{C.-K.} \bibnamefont{Tsung}},
  \bibinfo{author}{\bibfnamefont{H.}~\bibnamefont{Ho}},
  \bibinfo{author}{\bibfnamefont{S.}~\bibnamefont{Chou}},
  \bibinfo{author}{\bibfnamefont{J.}~\bibnamefont{Lin}}, \bibnamefont{and}
  \bibinfo{author}{\bibfnamefont{S.}~\bibnamefont{Lee}}, in
  \emph{\bibinfo{booktitle}{2016 International Computer Symposium (ICS)}}
  (\bibinfo{year}{2016}), pp. \bibinfo{pages}{12--17}.

\bibitem[{\citenamefont{Chung}(1997)}]{chung1997spectral}
\bibinfo{author}{\bibfnamefont{F.~R.~K.} \bibnamefont{Chung}},
  \emph{\bibinfo{title}{Spectral Graph Theory}}, no.~\bibinfo{number}{92} in
  \bibinfo{series}{CBMS Regional Conference Series in Mathematics}
  (\bibinfo{publisher}{American Mathematical Society},
  \bibinfo{address}{Providence, RI}, \bibinfo{year}{1997}).

\bibitem[{\citenamefont{Karush}(1939)}]{karush1939minima}
\bibinfo{author}{\bibfnamefont{W.}~\bibnamefont{Karush}},
  \emph{\bibinfo{title}{Minima of functions of several variables with
  inequalities as side constraints}}, \bibinfo{journal}{M. Sc. Dissertation.
  Dept. of Mathematics, Univ. of Chicago}  (\bibinfo{year}{1939}).

\bibitem[{\citenamefont{Kuhn and Tucker}(2013)}]{kuhn2013nonlinear}
\bibinfo{author}{\bibfnamefont{H.~W.} \bibnamefont{Kuhn}} \bibnamefont{and}
  \bibinfo{author}{\bibfnamefont{A.~W.} \bibnamefont{Tucker}},
  \emph{\bibinfo{title}{Nonlinear programming}}, in
  \emph{\bibinfo{booktitle}{Traces and emergence of nonlinear programming}}
  (\bibinfo{publisher}{Springer}, \bibinfo{year}{2013}), pp.
  \bibinfo{pages}{247--258}.

\bibitem[{\citenamefont{Livan et~al.}(2018)\citenamefont{Livan, Novaes, and
  Vivo}}]{Livan_2018}
\bibinfo{author}{\bibfnamefont{G.}~\bibnamefont{Livan}},
  \bibinfo{author}{\bibfnamefont{M.}~\bibnamefont{Novaes}}, \bibnamefont{and}
  \bibinfo{author}{\bibfnamefont{P.}~\bibnamefont{Vivo}},
  \emph{\bibinfo{title}{Introduction to Random Matrices}}
  (\bibinfo{publisher}{Springer International Publishing},
  \bibinfo{year}{2018}), ISBN \bibinfo{isbn}{9783319708850},
  \urlprefix\url{http://dx.doi.org/10.1007/978-3-319-70885-0}.

\end{thebibliography}

\clearpage

\begin{appendix}

\section{Transforming Integer Programming to a Standard MIQCQP}
\label{subsec:app_ip2miqcqp}

In Eq.~\eqref{eqn:miqcqp}, we showed a standard form of MIQCQP with binary decision variables. However, the decision variables in the optimization problem of Eq.~\eqref{eqn:quad_constraint_ps} are integers. Here, we illustrate how to use a binary expansion to convert the integer decision variables into binaries by adding additional constraints. 

Let us assume that each element of $\bm{x}$ is at most $m \in \mathbb{Z}^+$, and let $l$ be a positive integer such that $2^{l-1} \le m < 2^l$. For each $j$-th element of $\bm{x}$, $x_j$, we can assign $l$ binary decision variables $y_{jl}$ such that, 
\begin{equation}
x_j = \sum_{k=0}^{l-1} y_{jk} 2^k,\quad\mbox{for~} j=1\ldots n.
\label{eqn:above}
\end{equation}
We can then use Eq.~\eqref{eqn:above} to transform Eq.~\eqref{eqn:quad_constraint_ps} into Eq.~\eqref{eqn:miqcqp}. To do so, additional inequalities must be added in order to guarantee that each element of $\bm{x}$ is at most $m$, namely
$$
\sum_{k=0}^{l-1} y_{jk} 2^k \le m,\quad\mbox{for~} j=1\ldots n.
$$
The rest of the transformation follows straightforwardly, thus allowing us to solve an integer problem leveraging binary variables. 

\section{Analytical bound on the solutions with decomposed pipelines}
\label{app:analyticsl_1}
\renewcommand{\thesubsection}{\Alph{subsection}}

Here, we discuss the quality of solutions we obtain from the decomposition pipeline to solve the problems as in Eq.~\eqref{eqn:risk_aversed_po}. In particular, we derive a bound on the degradation in the objective function under the decomposition pipeline. From this bound, we can devise strategies to mitigate the degradation. 
For example, as in Algorithms~\ref{algorithm_modified_spectral} and \ref{algorithm_valid_split}, we show that the communities are constructed using the largest eigenvalue and eigenvector, and the number of communities is made to be as small as possible. These facts can be explained directly from the bound given in the following theorem.

\newtheorem{thm}{Theorem}

\begin{thm}[Gap between optimal direct and decomposed solutions]
\label{thm:solution_gap}
Let $H(\bm{x})$ be the objective function as in Eq.~\eqref{eqn:risk_aversed_po} for $\bm{x} \in \mathbb{R}^n$. Let $\Sigma$ be the covariance matrix and $\Sigma'$ be the block-diagonal matrix derived from $\Sigma$ using the decomposition pipeline. Let $\bm{x}^*$ and $\bm{x}'$ be the optimal assignments for the direct $\Sigma$ and decomposed $\Sigma'$, respectively. Let $K$ be the number of communities in the pipeline so that $\Sigma' \equiv \mbox{Diag}\left(\Sigma'_{11}, \ldots, \Sigma'_{K K}\right)$ and $\Gamma \equiv \min_k \| \Sigma_{kk} \|$. Then, it holds that
\begin{equation*}
    |H(\bm{x}^*) -{H}({\bm{x}}') | \le \frac{\|\Sigma^{-1}\|\cdot \| \bm{\mu} \|^2}{4q}~\left(\frac{K~\|\Sigma - \Sigma' \|}{\Gamma} \right)^2 .
\end{equation*}
\end{thm}
\begin{proof}
For completeness, we rewrite $H(\bm{x})$ as
\begin{equation}    
H(\bm{x})\equiv -\bm{\mu}^T \bm{x} + q \bm{x}^T \Sigma \bm{x},
\end{equation}
and ${H}'(\bm{x})$ as, 
\begin{equation}    
{H}'(\bm{x})\equiv -\bm{\mu}^T \bm{x} + q \bm{x}^T \Sigma' \bm{x}.
\end{equation}
In the limit of a large time horizon, the covariance matrices $\Sigma$ and $\Sigma'$ will both be positive definite, which can be seen from the form of the Marchenko-Pastur distribution, or from the fact that the Cholesky factors have sufficient dimensions to make the covariance matrices full rank. Then, by the convexity of the objective functions, one can see that $H(\bm{x})$ and  ${H}'(\bm{x})$ are optimized, respectively, at  
$$
\bm{x}^* = \frac{1}{2q} \Sigma^{-1} \bm{\mu},~\quad~\mbox{and}~\quad~{\bm{x}}' = \frac{1}{2q} (\Sigma')^{-1} \bm{\mu}.
$$
Thus, we can derive 
\begin{align*}
&|H(\bm{x}^*) - H(\bm{x}') |= \\
&=\frac{1}{4q}\left| \bm{\mu}^T\left[2\left(\Sigma'^{-1} - \Sigma^{-1} \right) + \left(\Sigma^{-1} - (\Sigma'^{-1})^T \Sigma \Sigma'^{-1}\right)\right]\bm{\mu} \right| \\
&= \frac{1}{4q}\left| \bm{\mu}^T\left(2\Sigma'^{-1} - \Sigma^{-1} - \Sigma'^{-1} \Sigma \Sigma'^{-1}\right)\bm{\mu} \right| \\
&= \frac{1}{4q}\left| \bm{\mu}^T\left(\Sigma'^{-1} - \Sigma^{-1} - \Sigma'^{-1} \Sigma \Sigma'^{-1} + \Sigma'^{-1}\right)\bm{\mu} \right| \\
&= \frac{1}{4q}\left| \bm{\mu}^T\left[\Sigma'^{-1}(\Sigma - \Sigma')\Sigma^{-1} - \Sigma'^{-1}(\Sigma-\Sigma') \Sigma'^{-1}\right]\bm{\mu} \right| \\
&= \frac{1}{4q}\left| \bm{\mu}^T \left[ \Sigma'^{-1}(\Sigma - \Sigma')(\Sigma^{-1} - \Sigma'^{-1})\right]\bm{\mu} \right|\\
&= \frac{1}{4q}\left| \bm{\mu}^T \left[ \Sigma'^{-1}(\Sigma - \Sigma')\Sigma^{-1}(\Sigma'-\Sigma) \Sigma'^{-1}\right]\bm{\mu} \right|\\
&\leq \frac{1}{4q} \|\bm{\mu}\|^2 ~ \|\Sigma'^{-1}\|^2 ~ \|\Sigma - \Sigma'\|^2 ~ \|\Sigma^{-1}\|,
\end{align*}
where we have used the fact that $\Sigma$ and $\Sigma'$ are positive definite to set $(\Sigma^{-1})^T=\Sigma^{-1}$ and similarly for $\Sigma'$.

Finally, the bound in the theorem is obtained by noticing that $\| (\Sigma')^{-1} \| \le \sum_{k} \|\left((\Sigma')_{kk}\right)^{-1} \| \le K/\Gamma$. 

\end{proof}

The bound in the theorem enables us to devise a decomposition strategy to mitigate the degradation of the solutions by choosing as few communities $K$ as possible, partition the communities so that each $\|\Sigma'_{kk}\|$ is large, and the matrix $\Sigma'$ is as close to $\Sigma$ as possible.

\section{Spectral Clustering Algorithm with Threshold of Community Size}
\label{app:alg_2}

We present the splitting function (Algorithm~\ref{algorithm_valid_split}) which is used in Algorithm~\ref{algorithm_modified_spectral}. We also discuss Algorithm~\ref{algorithm_modified_iterative} to find communities smaller than a given threshold on the size of the communities.

\begin{algorithm*}[tbh]
\caption{Valid Splitting Function}
\begin{algorithmic}[1]
\Function{valid\_split}{$G$}
    \State Compute the modularity matrix $Q_c$ for $G$ \Comment{See Eq.~\eqref{eqn:C_tot}}
    \State Compute the highest eigenvalue and eigenvector of $Q_c$
    \State \quad $I_{\text{pos}} = \{i \mid \text{eigvec}[i] \geq 0\}$, \;  $I_{\text{neg}} = \{i \mid \text{eigvec}[i] < 0\}$ 
    \If{$(I_{\text{pos}} \neq \emptyset \; \text{and} \; I_{\text{neg}} \neq \emptyset)$} 
        \State $G_1, G_2 = $ subgraph of $G$ with vertices in $I_{\text{pos}}$, $I_{\text{neg}}$ 
        \If{ $\Delta Q_c^{(G_1|G_2)}>0$ (See Eq.~\eqref{eqn:Qc_gain} ) } 
        \State \Return $(\text{true}, G_1, G_2)$ \Comment{$G$ can be split into $G_1$ and $G_2$}
        \Else
        \State \Return $(\text{false}, G, \emptyset)$ \Comment{$G$ cannot be split}

        \EndIf
    \Else
        \State \Return $(\text{false}, G, \emptyset)$
    \EndIf
\EndFunction
\end{algorithmic}
\label{algorithm_valid_split}
\end{algorithm*}

\begin{algorithm*}[tbh]
\caption{Modified Modularity-Based Spectral Optimization with Threshold}
\begin{algorithmic}[1]
\Require $C_{\mbox{struct}}$: Clean correlation matrix  (Eq.~\eqref{eqn:3_way_split}), $\phi$: threshold 
\Ensure $\bm{g}$: vector of community assignments 
\State Initialize: $\bm{g} \gets Algorithm~\ref{algorithm_modified_spectral} ~(C_{\mbox{struct}})$, and $U \gets $ subgraphs whose sizes exceed $\phi$ 

\While{$U \neq \emptyset$}
    \State $G' \gets \text{pop}(U) $  \Comment{Details in Sec.~\ref{sec:Clustering_with_Correlation_based_Modularity}}
    \If{$|G'| \leq \phi$} \Comment{Check if the community size is less than the threshold}
        \State Update $\bm{g}$ to reflect the community labels for the vertices in $G'$
    \Else    
        \State $(G_1, G_2) \gets \text{SPECTRAL\_SPLIT}(G')$ \Comment{Split $G'$}
        \State $U \gets U \cup \{G_1, G_2\}$

    \EndIf

\EndWhile
\State \Return $\bm{g}$: the final vector of community assignment to each node 
\vspace{5pt}

\Function{spectral\_split}{$G$}
    \State Compute the modularity matrix $Q^{*}_c$ for $G$. \Comment{See Eq.~\eqref{eqn:new_modularity}}
    \State Compute the highest eigenvalue and eigenvector of $Q^{*}_c$
        \State $G_1 \gets $ subgraph of $G$ with vertices in $I_{\text{pos}}=\{i \mid \text{eigvec}[i] \geq 0\}$
        \State $G_2 \gets $ subgraph of $G$ with vertices in $I_{\text{neg}}=\{i \mid \text{eigvec}[i] < 0\}$
        \State \Return $( G_1, G_2)$ \Comment{$G_1$ and $G_2$ are both non empty sets}

\EndFunction

\end{algorithmic}
\label{algorithm_modified_iterative}
\end{algorithm*}

\section{Rescaling factor in optimization problems with quadratic constraints}\label{sec:deriving_r}
\addtocontents{toc}{\protect\setcounter{tocdepth}{2}}

We argued that the rescaling factor $r$ is smaller than $1$ in Sec.~\ref{quad_cons} so that $a \le \bm{x}^T_b \Sigma \bm{x}_b$. Here, we give its explanation based on the continuous setting for the optimization problem Eq.~\eqref{eqn:quad_constraint_ps} described by a convex objective function with a quadratic constraint of the form
\begin{equation}
\begin{aligned}
& \underset{\bm{x} \in \mathcal{R}^n}{\min}
& & (\bm{x}-\bm{x}_b)^T \Sigma (\bm{x}-\bm{x}_b) \\
& \text{subject to}
& & \bm{x}^T \Sigma \bm{x} \leq a,
\end{aligned}
\label{eqn:quad_constraint_cont}
\end{equation}
The Lagrangian \( \mathcal{L}(x, \lambda) \) for this constrained optimization problem is defined as 
\begin{equation}
\mathcal{L}(x, \lambda) = (\bm{x}-\bm{x_{b}})^T \Sigma (\bm{x}-\bm{x}_b) + \lambda (\bm{x}^T \Sigma \bm{x} - a),
\end{equation}
where \( \lambda \) is the Lagrange multiplier associated with the inequality constraint. 
The Karush-Kuhn-Tucker (KKT) conditions~\cite{karush1939minima,kuhn2013nonlinear}, which provide necessary conditions for optimality, are given by:
\begin{widetext}
\begin{align*}
    \text{Primal Feasibility:} & \quad \bm{x}^T \Sigma \bm{x} \leq a,\\
    \text{Dual Feasibility:} & \quad \lambda \geq 0, \\
    \text{Stationarity:} & \quad \nabla_x \mathcal{L}(\bm{x}, \lambda) = 2\Sigma(\bm{x} - \bm{x}_b) + 2\lambda \Sigma \bm{x} = 0, \\
    \text{Complementary Slackness:} & \quad \lambda (\bm{x}^T \Sigma \bm{x} - a) = 0.
\end{align*}
\end{widetext}
To find the optimal solution, we first set the gradient of the Lagrangian with respect to \( \bm{x} \) to zero (stationarity condition) and solve for \( \bm{x} \) in terms of \( \lambda \) and \( \bm{x}_b \), i.e., 
\begin{equation}
(1 + \lambda) \Sigma \bm{x} = \Sigma \bm{x}_b \implies \bm{x} = \frac{\bm{x}_b}{1 + \lambda}.
\end{equation}
Substituting this expression into the primal feasibility condition \( \bm{x}^T \Sigma \bm{x} \leq a \), we solve for \( \lambda \):
\begin{equation}
\lambda \geq \sqrt{\frac{\bm{x}_b^T \Sigma \bm{x}_b}{a}} - 1 \ge 0,
\end{equation}
which implies that
$$
a \le \bm{x}_b^T \Sigma \bm{x}_b.
$$ 
For a minimization problem with ``$\le$'' constraint, we consider the positive square root in the computation of \( \lambda \). If the computed \( \lambda \) is negative, we set \( \lambda = 0 \), indicating that the constraint is not active at the optimal solution. The optimal solution to the optimization problem is obtained by substituting the value of \( \lambda \) back into the equation for \( \bm{x} \). The solution must satisfy all the KKT conditions to be considered valid.

\section{Numerical Experiments \label{appendix:numerics}}

\begin{figure}[tbh]
    \includegraphics[width=\columnwidth]{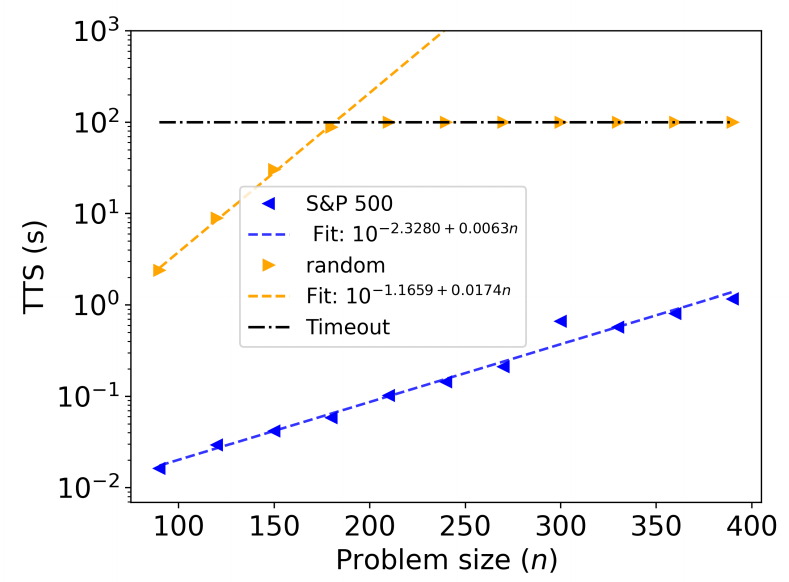}
    \includegraphics[width=\columnwidth]{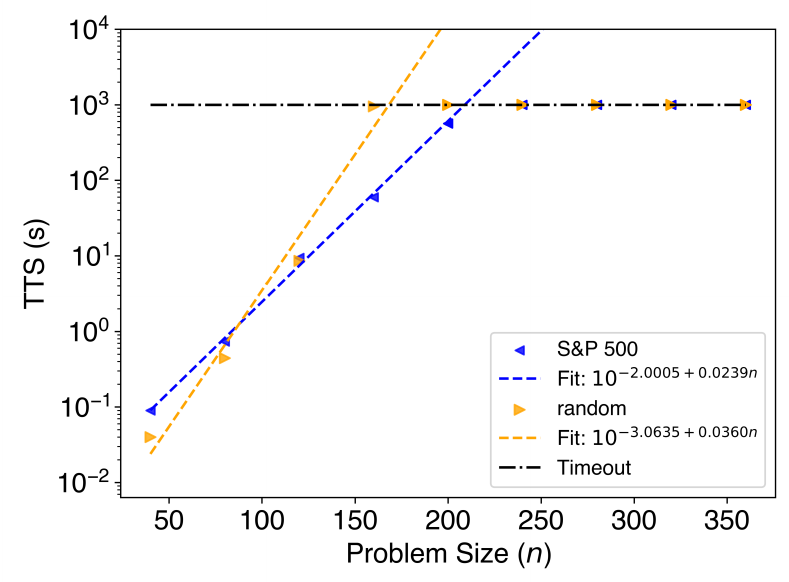}
    \caption{Median time-to-solution (TTS) over 10 seeds as a function of the problem size ($n$) for solving the portfolio-optimization problem with one cardinality constraint [Eq.~\ref{eqn:por_optim}, top panel] and a quadratic objective with a single quadratic constraint, as functions of problem size $n$ (bottom panel). Top panel: TTS scaling for random covariance (orange right triangles), $\propto 10^{0.00174n}$, is worse than for S\&P 500 (blue left triangles), $\propto 10^{0.0063n}$. The problems with random covariance reach the timeout of $100$ seconds at problem size $ n\sim 200$, but for the S\&P 500 data set we do not hit the timeout. Bottom panel: TTS scaling for random covariance (orange right triangles), $\propto 10^{0.036n}$, is worse than for the S\&P 500 data set (blue right triangles), $\propto 10^{0.0239n}$. The random covariance problem exceeds the 1000 second timeout at $n > 160$, while the S\&P 500 data set approaches the timeout at $n > 240$. The solver used in the experiments is Gurobi. 
    }
    \label{fig:TTS_quad_linear}
\end{figure}

We perform the optimization using Gurobi 11.0 from Gurobi Optimization (2023)~\cite{gurobi}. The hardware used is an Apple M3 MAX CPU with $16$ physical cores, $16$ logical processors, and usage of up to $16$ threads.

We also conducted additional numerical simulations beyond those discussed in the main text. In the results section, we presented findings on portfolio optimization with cardinality constraints set at $n/2$. Here, we provide additional results using a different set of cardinality constraints, specifically $n/3$, see Fig.~\ref{fig:TTS_po_1_n3}. Although the score gap is higher, the TTS is slightly elevated for all methods compared to our previous findings with constraints of $n/2$. Additionally, we plot the time taken by each component of the pipeline in Fig.~\ref{fig:decomp_sub}. Because both preprocessing and clustering rely on matrix diagonalization operations, their complexities are $\mathcal{O}(n^3)$. The total time for the optimization of subproblem (green right triangles in the figure) has a qualitatively similar trend to the other two parts of the decomposition pipeline. Using Fig.~\ref{fig:decomp_sub}, we find that all modules consume similar process times indicating that there is no bottleneck in the larger pipeline. 

\begin{figure}[tbh]
    \includegraphics[width=\columnwidth]{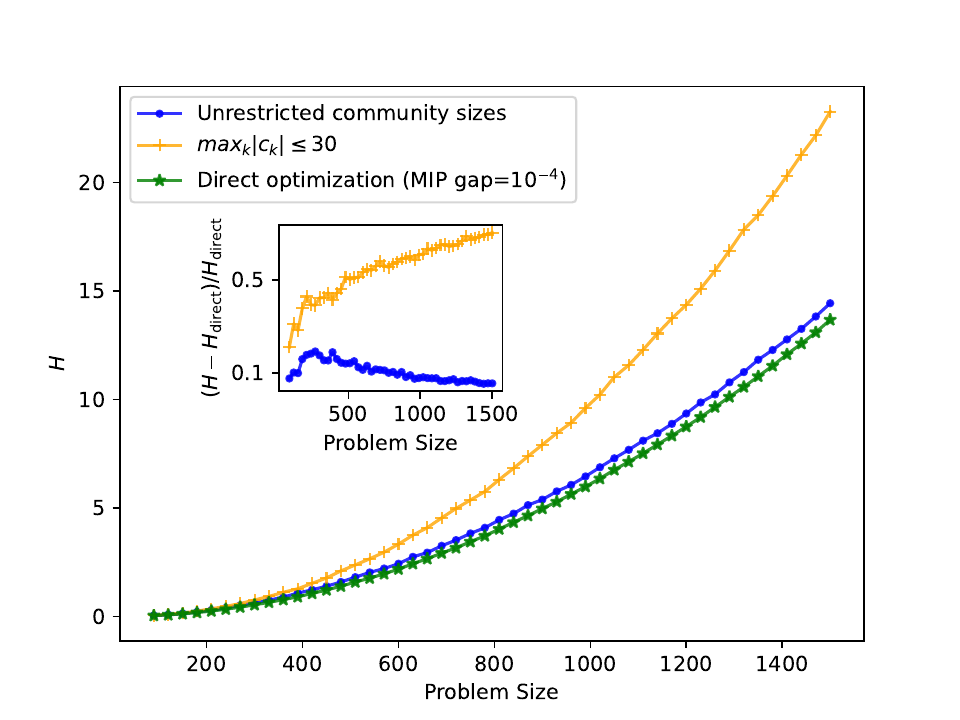}           
    \includegraphics[width=\columnwidth]{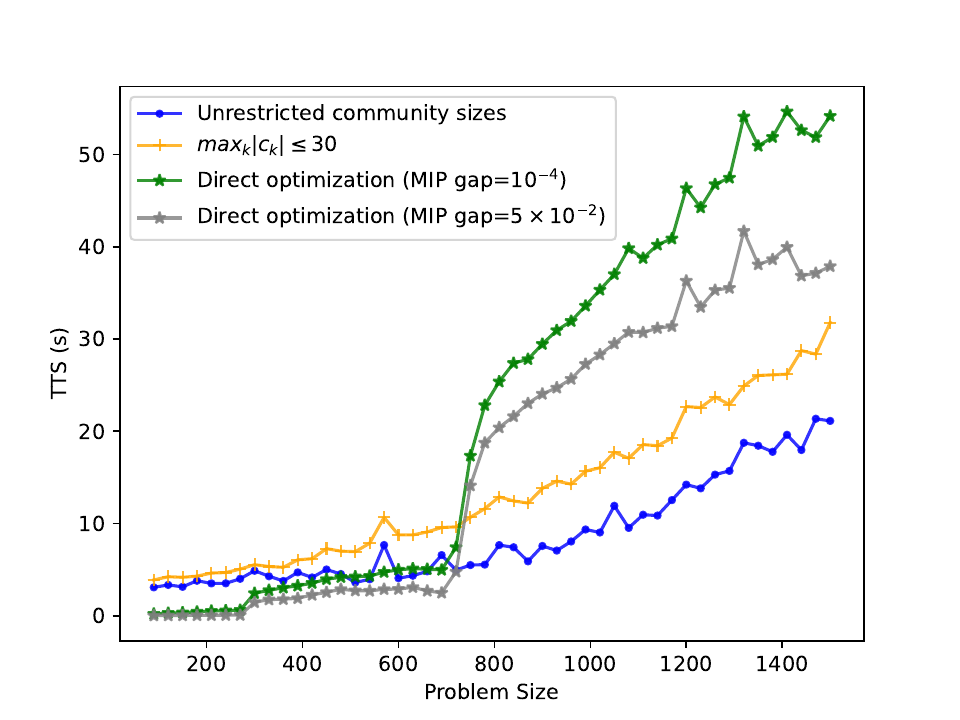}
    \caption{Performance as a function of the problem size of the PO problem with linear cardinality constraint set at $n/3$ (Eq.\eqref{eqn:por_optim}). Top panel: Solution quality $H$. Bottom panel: Time to solution (TTS). The line with green (grey) stars corresponds to directly optimizing with the MIP gap set at $5 \times 10^{-4}$ ($5 \times 10^{-2}$). The blue line with dots and the orange with plus signs refer to the proposed decomposition pipeline, where the former corresponds to an unrestricted community size and latter to restricting the size of the largest community size thresholded at $30$. The inset plot in the top panel shows the relative drop in objective for our decomposition pipeline with and without restrictions over the community size.}
\label{fig:TTS_po_1_n3}
\end{figure}

\begin{figure}[htb]
    \centering
    \includegraphics[width=9cm]{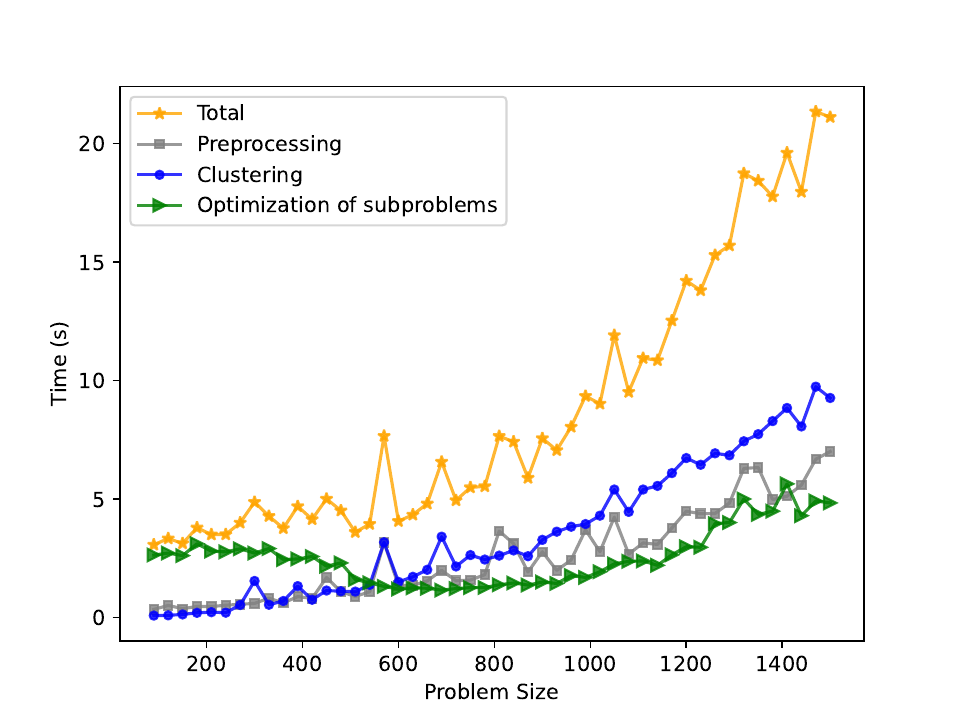}
    \caption{Time taken (process time) for each module (preprocessing, clustering, and optimization) of the decomposition pipeline. We plot median over $10$ seeds of solving for  portfolio optimization with assets taken from the Russell 3000 index and cardinality constraints set at $n/3$ as a function of the problem size $n$. Each random seed corresponds to adding incrementally a specific pool of assets from the lower limit of problem size $90$ to $1500$ in step sizes of $30$. The orange stars denote the total run time of the pipeline, grey squares correspond to the preprocessing, blue dots to the clustering step, and green right triangles to the optimization step.}
    \label{fig:decomp_sub}
\end{figure}

We also obtained benchmark results with CPLEX v20.1. However, for the integer portfolio optimization problem, the standard setting needs to be modified in order to obtain a reasonable time to solution. Specifically, the \texttt{qtolin} parameter \cite{cplex2009v12}---which controls the linearization of the quadratic terms in the objective function of a (mixed integer) quadratic program---needs to be turned off, because the default setting leads to a significant increase in the TTS. However, even with this setting disabled, our approach consistently obtained much better TTS results using Gurobi, which is why we used these numbers in this study. Additionally, for the optimization problems with a single quadratic constraint [Eq.~\eqref{eqn:quad_constraint_ps}], the CPLEX solver does not support the loading of the problem formulation itself unless it is first binarized. Due to the numerical instability of this approach, we again chose to only benchmark our decomposition pipeline against Gurobi.

\section{Derivation of Marchenko-Pastur Distribution}
\label{sec:app_RMT}

We begin with a short proof of Wigner's surmise using Stieltjes transformation \cite{Livan_2018}.
Let $W$ be an $n \times n$ symmetric random matrix whose entries on and above the diagonal are i.i.d.~with mean 0 and variance 1. Consider the scaled matrix $\frac{W}{\sqrt{n}}$, and let $\lambda_1, \ldots, \lambda_n$ be its eigenvalues. The Stieltjes transform of the empirical spectral distribution of $\frac{W}{\sqrt{n}}$ is given by 
\[ S_n(z) = \frac{1}{n} \sum_{i=1}^{n} \frac{1}{\lambda_i - z}. \]
Using the resolvent matrix, the Stieltjes transform can also be expressed as 
\[ S_n(z) = \frac{1}{n} \Tr \left[ \left(\frac{W}{\sqrt{n}} - zI\right)^{-1} \right]. \]
Let us denote by $M$ the matrix $W/\sqrt{n} - zI$, which can be partitioned as 
\[ M = \begin{bmatrix} \frac{W_{11}}{\sqrt{n}} - z & \frac{g^T}{\sqrt{n}} \\ \frac{g}{\sqrt{n}} & \tilde{M} \end{bmatrix}, \]
where $\tilde{M}$ is the $(n-1) \times (n-1)$ bottom right block of $M$ and $g$ is a column vector consisting of the off-diagonal elements of the first row of $W$. Applying the Schur complement formula for the inversion of block matrices, the $(1,1)$-entry of $M^{-1}$, denoted by $M_{11}^{-1}$, is given by 
\[ M_{11}^{-1} = \left(\frac{W_{11}}{\sqrt{n}} - z - \frac{g^T}{\sqrt{n}} \tilde{M}^{-1} \frac{g}{\sqrt{n}}\right)^{-1}. \]
As $n$ approaches infinity, the term
$$
\frac{g^T}{\sqrt{n}} \tilde{M}^{-1} \frac{g}{\sqrt{n}}
$$ 
concentrates around its expectation due to the law of large numbers. Hence, we can approximate $M_{11}^{-1}$ by 
$$\left[\frac{W_{11}}{\sqrt{n}} - z - \frac{1}{n} \Tr(\tilde{M}^{-1})\right]^{-1},
$$ 
which simplifies to $-z - S_n(z)$, assuming that the expectation of $\Tr(\tilde{M}^{-1})$ is close to $nS_n(z)$. Thus, the Stieltjes transform satisfies the equation 
\[ S_n(z) \approx \frac{1}{-z - S_n(z)}. \]
Solving this equation yields the Stieltjes transform of the semicircle law:
\[ S(z) = \frac{-z \pm \sqrt{z^2 - 4}}{2} \]

To derive the Marchenko Pastur (MP) distribution \cite{1967SbMat...1..457M}, let \( X_{d\times 1} \) be a random vector uniformly distributed over the eigenvalues \( \{\lambda_1, \ldots, \lambda_d\} \). Define the sample covariance matrix 
$$\Sigma_n = \frac{1}{n} \sum_{i=1}^n X_i X_i^T,
$$ 
which is a \( d \times d \) symmetric matrix. The Stieltjes transform of the empirical spectral distribution of \( \Sigma_n \) is:
\begin{equation}
S_n(z) = \frac{1}{d} \Tr \left( \Sigma_n - zI \right)^{-1}  \;  z \in \mathbb{C}  \setminus \mathbb{R}.
\end{equation}  
To calculate \( S_n(z) \), we consider the quantity \(  \sum_{i=1}^n X_i X_i^T - nzI \), which we denote by \( A \). Thus, we have $ S_n(z) = (n/d)\Tr \left( A^{-1} \right) $. Now, If we assume there exists a limiting spectral distribution as \( n \) grows large, we can approximate \( S_n(z) \) through \( S_{n-1}(z) \). Let \( B =  \sum_{i=1}^{n-1} X_i X_i^T - nzI \) and $\beta=(d/n)$. Then:
\begin{equation}
S_{n-1}(z) = \frac{n-1}{d} \Tr \left( B^{-1} \right) \approx \frac{1}{\beta} \Tr(B^{-1})
\end{equation} 
We separate out the addition of the \( n \)-th term as follows: 
\[ A^{-1} = \left[ B + X_n X_n^T \right]^{-1}. \]
Now, considering the independence of \( X_n \) from \( B \), we use the Sherman-Morrison formula to obtain
\begin{equation}
\label{eq:Sherman_morrison_MP}
X_n^{T} A^{-1} X_n = X^{T}_n (B + X_n^T X_n)^{-1} X_n = 1-\frac{1}{1 + X_n^T B^{-1} X_n},
\end{equation}  
which allows us to connect \( S_n(z) \) and \( S_{n-1}(z) \) through the relationship between \( A^{-1} \) and \( B^{-1} \). We now have the basic relationship for a recursive calculation of the Stieltjes transform as \( n \) becomes large. From the independence of $X_n$ and $B$, and conditioning on $B$, we can obtain the following for the denominator in Eq.~\eqref{eq:Sherman_morrison_MP}
\begin{equation}
 X_n^T B^{-1} X_n \approx \mathbb{E}[X_n^T B^{-1} X_n] = \Tr(B^{-1}) \approx \Tr(A^{-1}),
\end{equation}
where the approximation follows from concentration considerations, and where we have also used the relationship $\Tr(A^{-1}) \approx \Tr(B^{-1})$. We therefore have the following approximation:
\begin{equation}
X_n^T A^{-1} X_n \approx 1 - \frac{1}{1 + \Tr(A^{-1})}    
\end{equation}
Similarly, for a different slicing of the matrix, $X_k$, we assume that the same approximation holds. Summing over $n$ slices, we have:
\begin{equation}
\sum_{k=1}^{n} X_k^T A^{-1} X_k \approx n \left[1 - \frac{1}{1 + \Tr(A^{-1})}\right]    
\end{equation}
Using the linearity and cyclic property of the trace, we can rewrite the sum as 
$$
\Tr\left[\sum_{k=1}^{n} X_k^T A^{-1} X_k\right] = \Tr\left[\sum_{k=1}^{n} X_k X_k^T A^{-1}\right].
$$
Because $\sum_{k=1}^{n} X_k X_k^T = (A + nzI)$, we obtain
\begin{equation}    
\Tr\left[(A + znI)A^{-1}\right] = d + zn \, \Tr(A^{-1}).
\end{equation}
We have shown that
\begin{equation}
d + zn\Tr(A^{-1}) \approx n \left[ 1 - \frac{1}{1 + \Tr(A^{-1})} \right].
\end{equation}
This is a quadratic equation in $\Tr(A^{-1})$. To solve it, we use $s_n(\Sigma) = (1/\beta) \Tr(A^{-1})$ and find:
\begin{equation}
s_n(z) = \frac{1 - z - \beta + \sqrt{(1 -z-\beta)^2-4z\beta}}{2}    
\end{equation}
This is the same as for the Stieltjes transformation of the Marchenko Pastur distribution \cite{Livan_2018}, which demonstrates uniqueness and hence proves the result that, indeed, the eigenvalue distribution of random correlation matrices follow the Marchenko Pastur distribution. The density function of the Marchenko Pastur distribution is given by:
\[
\rho_{MP}(\lambda) = \frac{1}{2\pi \beta \lambda} \sqrt{(\lambda_+ - \lambda)(\lambda - \lambda_-)}, \quad \lambda \in [\lambda_-, \lambda_+],
\]
where \(\lambda_{\pm} = (1 \pm \sqrt{\beta})^2\).

The Stieltjes transform \(S(z)\) of a probability distribution \(\rho(\lambda)\) is defined as
\[
S(z) = \int \frac{\rho(\lambda)}{\lambda - z} \, d\lambda, \quad z \in \mathbb{C} \setminus \mathbb{R}.
\]
For the Marchenko Pastur distribution, the Stieltjes transform \(S_{MP}(z)\) can be derived as follows:
   \[
   S_{MP}(z) = \int_{\lambda_-}^{\lambda_+} \frac{\rho_{MP}(\lambda)}{\lambda - z} \, d\lambda
   \]
Substituting \(\rho_{MP}(\lambda)\), we have
   \begin{equation}
   S_{MP}(z) = \int_{\lambda_-}^{\lambda_+} \frac{1}{2\pi \beta \lambda} \sqrt{(\lambda_+ - \lambda)(\lambda - \lambda_-)} \frac{1}{\lambda - z} \, d\lambda.       
   \end{equation}
Using the Residue Theorem, this integral evaluates to
\begin{equation}
S_{MP}(z) = \frac{1 - \beta - z + \sqrt{(z - \lambda_-)(z - \lambda_+)}}{2 \beta z}.
\end{equation}
Here, the square root is chosen such that \(\sqrt{(z - \lambda_-)(z - \lambda_+)}\) has a branch cut on the real axis from \(\lambda_-\) to \(\lambda_+\). We have shown the derivation of the Marchenko Pastur distribution for correlation matrices with $\sigma^2=1.$

\end{appendix}

\end{document}